\documentclass{amsart}

\usepackage{amsmath}
\usepackage{amssymb}
\usepackage{hyperref}
\usepackage{cases}
\usepackage{cleveref}
\usepackage{xcolor}

\newtheorem{theorem}{Theorem}[section]
\newtheorem{prop}[theorem]{Proposition}
\newtheorem{lemma}[theorem]{Lemma}
\newtheorem{cor}[theorem]{Corollary}

\theoremstyle{definition}

\theoremstyle{remark}
\newtheorem{remark}[theorem]{Remark}

\DeclareMathOperator\supp{supp}

\DeclareMathOperator\tr{tr}
\def\ep{\varepsilon}    %%    then \ep gives the nice-looking epsilon

\newcommand{\R}{\mathbb{R}}

\begin{document}

\title{Extrinsic bi-Conformal Heat Flow and its smoothness}

\author{Woongbae Park}
%\address{301 Thackeray Hall, Department of Mathematics, University of Pittsburgh, Pittsburgh, Pennsylvania, 15260}
%\email{wop5@pitt.edu}

\subjclass[2020]{Primary 58E20, 53E99, 35K91}

\keywords{biharmonic map flow, bi-CHF, global existence, smooth solution}

\begin{abstract}
In this paper we introduce conformal heat flow of (extrinsic) biharmonic maps on $4$-manifold, simply called bi-conformal heat flow (bi-CHF), and study its properties.
Similar to other CHF of harmonic maps and regularized $n$-harmonic maps, (CHF and regularized $n$-CHF respectively), we obtain global smoothness and no finite time singularity.
\end{abstract}

\maketitle
\sloppy

%%%%%%%%%%%%%%%%%%%%%%%%%%%%%%%%%%%%%%%%%%%%%%%%%%%%%%%%%%%%%%%%%%%%%%%%%%%%%
\section{Introduction}
\label{sec1}

Let $(M,g)$ and $(N,h)$ be two Riemannian manifolds with $\dim M = 4$.
The (extrinsic) biharmonic map flow is a gradient flow of (extrinsic) bienergy
\begin{equation}
E(f) = \frac{1}{2}\int_{M} |\Delta f|^2_g dvol_g.
\end{equation}
%It is well known that the bi-energy is conformally invariant in dimension $4$.

Its Euler-Lagrange equation becomes
\begin{equation}
 -\Delta^2 f + \Delta (A(df,df)) - \langle \Delta f, \Delta P \rangle + 2 \nabla \langle \Delta f, \nabla P \rangle = 0
\end{equation}
where $A$ is the second fundamental form of the embedding $(N,h) \hookrightarrow \mathbb{R}^L$ and $P$ is the orthogonal projection from $\mathbb{R}^L$ onto $T_f N$.
For simplicity we denote
\begin{equation} \label{B perp}
B =  \Delta (A(df,df)) - \langle \Delta f, \Delta P \rangle + 2 \nabla \langle \Delta f, \nabla P \rangle \perp T_f N.
\end{equation}
Hence the (extrinsic) biharmonic map equation becomes
\[
-\Delta^2 f \perp T_f N.
\]
%Using orthonormal frame $\nu^k$ for the normal space of $N \hookrightarrow \mathbb{R}^L$, each term in \eqref{B perp} can be rewritten as
%\[
%\begin{split}
%\Delta (A(df,df)) =& \Delta \langle df, d\nu^k(df) \rangle \nu^k\\
%\langle \Delta f, \Delta P \rangle =& \langle \Delta f, \Div (d\nu^k(df)) \rangle \nu^k + \langle \Delta f, d\nu^k (df) \rangle d\nu^k(df)\\
%\langle \nabla \Delta f, \nabla P \rangle =& \langle \nabla \Delta f, d\nu^k(df) \rangle \nu^k.
%\end{split}
%\]
Corresponding gradient flow equation becomes
\begin{equation}
f_t = -\Delta^2 f + \Delta (A(df,df)) - \langle \Delta f, \Delta P \rangle + 2 \nabla \langle  \Delta f, \nabla P \rangle.
\end{equation}

%                  biharmonic map flow also forms bubbling in finite time!!
%\bigskip
%History of research about biharmonic maps and their flows.

Biharmonic map is considered as an extension of harmonic maps, as the (extrinsic) bienergy can be considered as a higher dimensional generalization of Dirichlet energy
\begin{equation}
E_1(f) = \frac{1}{2} \int_{M} |df|^2 dvol_g.
\end{equation}
Note that there is another version of generalization, called intrinsic bienergy, defined by
\begin{equation}
E_2(f) = \frac{1}{2} \int_{M} |\tau(f)|^2 dvol_g
\end{equation}
where $\tau(f) = \tau_g (f) \in T_f N$ is the tension field of $f$.
Its critical point is called (intrinsic) biharmonic map.
%, whose Euler-Lagrange equation is given by
%\begin{equation}
%\begin{split}
%-\Delta^2 f + \Delta (A(df,df)) - \langle \Delta f, \Delta P \rangle + 2 \nabla \langle \Delta f, \nabla P \rangle &\\
%+ P (A(df,df) \cdot DA(df,df)) + 2 A(df,df) A(df, \nabla P) & = 0.
%\end{split}
%\end{equation}
Here we note that 
%(intrinsic or extrinsic) biharmonic maps can be written in local coordinates, or using the embedding $(N,h) \hookrightarrow \mathbb{R}^L$.
%Because extrinsic biharmonic maps depend on the embedding, many literatures about extrinsic biharmonic maps use a fixed embedding.
extrinsic biharmonic map depends on the embedding, whereas intrinsic biharmonic map does not.

%On the other hand, intrinsic biharmonic maps does not depend on the embedding, hence the equation for intrinsic biharmonic maps is described in terms of curvature tensor in many cases.
There are many results on biharmonic map including Jiang \cite{J86}, Chang-Wang-Yang \cite{CWY99}, Wang \cite{Wang04a}, \cite{Wang04b}, \cite{Wang04c}, Lamm-Rivi{\`e}re \cite{LammR08}, Hornung-Moser \cite{HM12}, Laurain-Rivi{\`e}re \cite{LaurainR13}, Fardoun-Saliba \cite{FS21}, and many others.
Due to the similar structure of being critical point of energy functional, biharmonic maps shares similar properties with harmonic map and $n$-harmonic map.
And such analogues hold for heat flow of corresponding elliptic systems.
Biharmonic heat flow was widely studied by many authors, including Lamm \cite{Lamm04}, \cite{Lamm05}, Moser \cite{M06}, Wang \cite{Wang07}, Hineman-Huang-Wang \cite{HHW14}, Liu-Yin \cite{LY15}, Laurain-Lin \cite{LL21}, and many others.
Among them, Wang \cite{Wang07} showed existence of global weak solution of extrinsic biharmonic map flow, while finite time singularities were obtained by Liu-Yin \cite{LY15}.
Under further condition is assumed, global smoothness was also obtained, with non-positive curvature of the target manifold by Lamm \cite{Lamm05} and with small initial energy by Lamm \cite{Lamm04}.

There was a new approach to obtain global smooth solution by varying domain metric in a suitable way.
Park \cite{P22} studied a variation of harmonic map flow by combining the flow with the evolution of the metric in conformal direction, called the conformal heat flow of harmonic maps (or simply, conformal heat flow or CHF) given by
\begin{equation}
\begin{cases}
f_t =& \tau_g (f)\\
u_t =& b |df|_g^2 - a
\end{cases}
\end{equation}
where $a,b$ are positive constants with $b$ large enough, $\tau_g = \tr_g (\nabla^g df)$ is the tension field of $f$ with respect to $g$ and $g = g(x,t) = e^{2u(x,t)}g_0(x)$ is the time-dependent metric of $M$ with conformal factor $u$.
The evolution equation of the conformal factor $u$ is designed to postpone any finite time bubbling.
Thanks to the conformal invariance of the energy, the CHF enjoys many properties similar to the harmonic map flow, including $\ep$-regularity and energy decreasing.
And no finite time bubbling occurs in CHF was shown recently in \cite{P23}.
This idea is robust in the sense that it can be applied to another conformally invariant energy concentrating geometric flow, (regularized) $n$-harmonic map flows, as shown in \cite{P25}.

In this regard, we seek for a similar variation of biharmonic map flow to overcome finite time singularity.
Unlike harmonic map equation, biharmonic map equation is not fully conformally invariant - it is invariant only under constant dilation.
%Nevertheless, we set up the evolution equation in the similar manner.
%Denote $g = e^{2u}g_0$ is the time-dependent metric with $u=u(x,t) : M \times [0,T] \to \mathbb{R}$ and $g_0$ is the initial time-independent metric.
Therefore, we do not set up the evolution equation in terms of time-dependent metric $g = e^{2u}g_0$ which is complicated.
Instead, we set up the equation in a simpler version, only multiplied by the conformal factor $e^{-4u}$.
Denote $g_0$ the time-independent metric of $M$ and $u = u(x,t) : M \times [0,T] \to \mathbb{R}$ be a conformal factor.
Consider a pair of equations
\begin{equation} \label{eq1} %\tag{bi-CHF}
\begin{cases}
f_t &= e^{-4u} \left( -\Delta^2 f + \Delta (A(df,df)) - \langle \Delta f, \Delta P \rangle + 2 \nabla \langle \Delta f, \nabla P \rangle \right)\\
u_t &= b e^{-4u} (|\nabla df|^2 + |df|^4) - a
\end{cases}
\end{equation}
with respect to $g_0$ with initial condition $f(0) = f_0 : M \to N \hookrightarrow \mathbb{R}^L$ and $u(0)=0$.
Call this equation by (extrinsic) bi-CHF.
We sometimes use the modified version instead, where we multiply both sides by $e^{4u}$.

\begin{remark}
In CHF case, the equation of $u$ looks like $u_t = b e_g(f) - a$ where $e_g(f)$ is the energy density with respect to $g$ from the energy $E(f) = \int e_g(f) dvol_g$.
This suggests that for bi-CHF case, the equation of $u$ is supposed to be $u_t = b |\Delta_g f|^2 - a$ or $e^{4u} u_t = b |\Delta f|^2 - ae^{4u}$.
But due to the lack of local control of $|df|^4$ and $|\nabla df|^2$, we change the energy density term in the equation of $u$ so that the metric react suitably not where $|\Delta f|^2$ is concentrated, but where $|\nabla df|^2 + |df|^4$ is concentrated.
\end{remark}

\begin{remark}
While we do not have enough a priori control of $|\nabla df|^2$ and $|df|^4$ locally, we have global control of them.
First, by integration by parts, $\int_M |\nabla df|^2 = \int_{M} |\Delta f|^2$ as usual.
Since we assume $N$ be a compact Riemannian manifold and $f : M \to N$, we can assume $\|f\|_{L^\infty} < \infty$.
In this regard, and with Sobolev embedding and elliptic regularity, we have
\begin{equation} \label{df4 bdd}
\int_{M} |df|^4 \leq C \left( \int_{M} |\Delta f|^2 + |f|^2 \right)^2 \leq C_1 \left( \int_{M} |\Delta f|^2 \right)^2 + C_2.
\end{equation}
Hence, $\int_M |df|^4$ is controlled by the energy $E(f)$.
\end{remark}

The main theorem of this paper is the following.
\begin{theorem} \label{main 1}
Let $f_0 \in W^{6,2}(M,N)$.
Then there exists a smooth solution $(f,u)$ of \eqref{eq1} on $M \times [0,\infty)$ with initial condition $f(0)=f_0, u(0)=0$.
\end{theorem}

Throughout this paper, we assume $(N,h) \hookrightarrow \R^L$ isometrically, and there is a constant $C_N$ such that
$\|A\|, \|DA\|, \|D^2 A\|, \|P\|, \|\nabla P\|, \|\nabla^2 P\| \leq C_N$.
%$\|{}^N\!\!R\|, \|D {}^N\!\!R\|,\|A\|, \|DA\| \leq C_N$ where ${}^N\!\!R$ is the Riemannian curvature tensor of $N$.
%Also assume $b$ is large enough so that
%\begin{equation}
%C_b := 4b - C_1 > 0.
%\end{equation}
%Also, we assume all the computations are made with respect to the metric $g_0$ unless the metric is specified.

%%%%%%%%%%%%%%%%%%%%%%%%%%%%%%%%%%%%%%%%%%%%%%%%%%%%%%%%%%%%%%%%%%%%%%%%%%%%%
\section{Preliminary}
\label{sec2}

In this section we recall some useful facts that are needed in the later sections.
In many cases, the computations require so-called commutator identities.
These are standard and can be found in many literatures, for example, see Lamm and Muller.

Let $B$ be a $(p,q)$ tensor on $M$, that is, a smooth section of $(T^\ast M)^{\otimes p} \otimes (TM)^{\otimes q}$.
Let $(x^1, \cdots, x^4)$ be a local coordinate in $U \subset M$.
Then we have
\begin{equation} \label{com 1}
[\nabla_i, \nabla_j] B^{k_1 \ldots k_q}_{\ell_1 \ldots \ell_p} = \sum_{r=1}^{q} R_{ijm}^{k_r} B^{k_1 \ldots k_{r-1} m k_{r+1} \ldots k_q}_{\ell_1 \ldots \ell_p} + \sum_{s=1}^{p} R_{ij\ell_s m} B^{k_1 \ldots k_q}_{\ell_1 \ldots \ell_{s-1} m \ell_{s+1} \ldots \ell_p}
\end{equation}
where $R$ is the Riemmian curvature tensor of $M$.

We can extend it along a map $f : M \to N$.
Together with the embedding $N \hookrightarrow \mathbb{R}^L$, we can only consider the map $f : M \to \mathbb{R}^L$.
Levi-Civita connection on $M$ and $\mathbb{R}^L$ can induce connections on $f^\ast TN \hookrightarrow f^\ast T \mathbb{R}^L$ and higher tensor powers, and above commutator identities still hold accordingly.

%We can extend it along a map $f : M \to N$.
%Suppose we have Levi-Civita connections on $M$ and $N$.
%It can induce connections on $f^\ast TN$ and higher tensor powers as well.
%Now let $x^i$ be a local coordinate on $M$ with $i=1, \ldots, 4 = \dim M$, $y^\mu$ be a local coordinate on $N$ with $\mu = 1, \ldots, n = \dim N$.
%Note that $f_\ast \frac{\partial}{\partial x^i} = \nabla_i f^\mu \frac{\partial}{\partial y^\mu}$.

Now let $x^i$ be a local coordinate on $M$ with $i=1, \ldots, 4$ and $\alpha = 1, \ldots, L$.
For simplicity, we denote $\nabla_i f^\alpha = f^\alpha_i, \frac{\partial}{\partial t} f^\alpha = f^\alpha_t$.
%Then \eqref{com 1} extends to mixed tensors using mixed curvature tensor, for example,
%\[
%R_{ij \mu \lambda} = \langle R (\frac{\partial}{\partial x^i},\frac{\partial}{\partial x^j}) f^\ast \frac{\partial}{\partial y^\mu}, f^\ast \frac{\partial}{\partial y^\lambda} \rangle = {}^N \!\! R_{\alpha \beta \mu \lambda} \nabla_i f^\alpha \nabla_j f^\beta
%\]
%where we denote ${}^N \!\! R $ for the Riemannian curvature tensor for $N$.
Here we illustrate a few identities commonly used later.
\begin{equation} \label{com 2}
\nabla_i \nabla_j \nabla_k f^\alpha = \nabla_j \nabla_i \nabla_k f^\alpha + R_{ijk \ell}  f^\alpha_\ell .
%+ {}^N \!\! R_{\mu \nu \beta}^{\alpha} f^\mu_i f^\nu_j  f^\beta_k.
\end{equation}
On $M \times [0,T]$, $\frac{\partial}{\partial t}$ induces $\nabla_t$ in usual sense over all bundles on $M \times [0,T]$.
Then in particular, we have
\begin{align}
\nabla_i \nabla_j  f^\alpha_t =& \nabla_j \nabla_i  f^\alpha_t
% + {}^N\!\! R^{\alpha}_{\mu \nu \beta}  f^\mu_i  f^\nu_j  f^\beta_t
 \label{com 3}\\
\nabla_t \left( \nabla_i \nabla_j f^\alpha \right) =& \nabla_j \nabla_i  f^\alpha_t 
%+ {}^N \!\! R^\alpha_{\mu \nu \beta} f^\mu_t  f^\nu_j f^\beta_i
 \label{com 4}\\
\nabla_i \nabla_j \nabla_k f^\alpha_t =& \nabla_j \nabla_i \nabla_k f^\alpha_t + R_{ij k \ell}\nabla_\ell  f^\alpha_t .
%+ {}^N \!\! R^{\alpha}_{\mu \nu \beta}   f^\mu_i  f^\nu_j \nabla_k f^\beta_t 
\label{com 5}
%\begin{split}
%\nabla_i \nabla_j \nabla_k f^\alpha_t =& \nabla_i \nabla_k \nabla_j  f^\alpha_t 
%%+ \nabla_i \left( {}^N \!\! R^{\alpha}_{\mu \nu \beta}  f^\mu_j f^\nu_k f^\beta_t \right)
%\\
%=& \nabla_i \nabla_k \nabla_j  f^\alpha_t + {}^N\!\! R^{\alpha}_{\mu \nu \beta, \gamma}  f^\gamma_i  f^\mu_j f^\nu_k f^\beta_t \\
%&+ {}^N\!\! R^{\alpha}_{\mu \nu \beta} \left( \nabla_i \nabla_j f^\mu f^\nu_k  f^\beta_t + f^\mu_j \nabla_i \nabla_k f^\nu  f^\beta_t + f^\mu_j f^\nu_k \nabla_i f^\beta_t \right) \label{com 6}.
%\end{split}
\end{align}

\begin{remark}
Above identities are true under the embedding $(N,h) \hookrightarrow \mathbb{R}^L$. 
Without embedding, the commutator identity will produce extra curvature terms in $N$.
Let $\dim N = n$ and choose a local coordinate $\{f^{\tilde{\alpha}}\}$ of $N$ with $\tilde{\alpha} = 1, \ldots, n$.
Then, for example,
\[
\nabla_t (\nabla_i \nabla_j f^{\tilde{\alpha}}) = \nabla_i \nabla_j f^{\tilde{\alpha}}_t + {}^N\!\! R^{\tilde{\alpha}}_{\tilde{\mu} \tilde{\nu} \tilde{\beta}} f^{\tilde{\mu}}_t f^{\tilde{\nu}}_i f^{\tilde{\beta}}_j
\]
where ${}^N\!\!R$ is the Riemannian curvature tensor of $N$.

Hence, under the embedding $(N,h) \hookrightarrow \mathbb{R}^L$, $\nabla_t$ commutes with other derivatives.
\end{remark}

Next, we provide simple observations about bi-CHF.

\begin{lemma} \label{E dec}
(Energy decreasing)
Let $(f,u)$ be a smooth solution of \eqref{eq1} on $M \times [0,T)$.
Then
\begin{equation}
\frac{d}{dt} E(f(t)) = -\int_{M} e^{4u} |f_t|^2 \leq 0.
\end{equation}
In particular, the energy is non-increasing and $E(t) \to E_\infty$ as $t \to \infty$.
\end{lemma}

\begin{proof}
By direct computation,
% and using \eqref{com  4},
and using the equation $e^{4u} f_t + \Delta^2 f = B \perp T_f N$,
\[
\frac{d}{dt} E(f(t)) = \int_{M} \langle \partial_t \Delta f, \Delta f \rangle = \int_{M} \langle \Delta f_t, \Delta f \rangle = \int_{M} \langle f_t, \Delta^2 f \rangle = -\int_{M} e^{4u} |f_t|^2 \leq 0.
\]
\end{proof}

\begin{lemma}
Let $(f,u)$ be a smooth solution of \eqref{eq1} on $M \times [0,T)$.
Then
\begin{equation} \label{e 4u}
e^{4u} = e^{-4at} \left( 1 + 4b \int_{0}^{t} e^{4as} (|\nabla df|^2(s) + |df|^4(s)) ds \right).
\end{equation}
Hence the volume $V(t) = \int_{M} dvol_g = \int_{M} e^{4u}$ satisfies
\begin{equation}
V(t) \leq e^{-4at} V(0) + \frac{b}{a} \left( 2 E(0) + 4C_1 E(0)^2 + C_2 \right).
\end{equation}
\end{lemma}

\begin{proof}
The second equation of \eqref{eq1} can be solved directly to get \eqref{e 4u}.
Using above lemma and equation \eqref{df4 bdd}, Integrating over $M$ will give
\[
\begin{split}
V(t) =&  e^{-4at} V(0) + 4b e^{-4at} \int_{0}^{t} e^{4as} \int_{M} (|\nabla df|^2(s) + |df|^4(s)) ds \\
\leq& e^{-4at} V(0) + 4b e^{-4at} \int_{0}^{t} e^{4as} \left( 2 E(0) + 4C_1 E(0)^2 + C_2 \right)\\
\leq& e^{-4at} V(0) + \frac{b}{a} \left( 2 E(0) + 4C_1 E(0)^2 + C_2 \right).
\end{split}
\]
\end{proof}

Note that, from \eqref{e 4u}, we can have for any $c>0$,
\[
e^{-cu} \leq e^{cat}.
\]

\begin{lemma} \label{e^pu}
($e^{pu}$ estimate)
Let $u$ be a solution of the second equation of \eqref{eq1}.
For any $p>1$ and for any $m \geq 0$ and $t_0 \leq t$,
\[
\int_{M} e^{4pu} \varphi^m (t) \leq \int_{M} e^{4pu} \varphi^m (t_0) + \frac{4(p-1)b^2}{pa}\int_{t_0}^{t} \int_{M}  \left( |\nabla df|^2 + |df|^4 \right)^{p} \varphi^m.
\]
\end{lemma}

\begin{proof}
By direct computation,
\[
\begin{split}
\frac{d}{dt} \left( \int_{M} e^{4pu} \varphi^m (t) \right) =& 4p \int_{M} e^{4pu} u_t \varphi^m (t)\\
=& 4p \int_{M} e^{4(p-1)u} \left( b (|\Delta f|^2 + |df|^4) - a e^{4u} \right) \varphi^m\\
\leq& 4(p-1)b \lambda \int_{M} e^{4pu} \varphi^m + 4b \lambda^{-1} \int_{M} \left( |\nabla df|^2 + |df|^4 \right)^{p} \varphi^m\\
& - 4pa \int_{M} e^{4pu} \varphi^m\\
=& 4b \lambda^{-1} \int_{M} \left( |\nabla df|^2 + |df|^4 \right)^{p} \varphi^m
\end{split}
\]
if we choose $\lambda = \frac{pa}{(p-1)b}$.
The proof is complete if we integrate it over $[t_0,t]$.
\end{proof}

Now we derive integration by parts type inequalities.

\begin{prop} \label{Int by parts general}
Let $F \in C^{\infty}(M \times [0,T],N)$ and $\psi \in C^{\infty}(M)$.
Then
\begin{equation}
\int_{M} |\nabla^2 F|^2 \psi^2 \leq 2 \int_{M} |\Delta F|^2 \psi^2 + C \int_{M} |\nabla F|^2 |\nabla \psi|^2 + C \int_{M} |\nabla F|^2 \psi^2
\end{equation}
where $C$ depends on the Ricci curvature of $g_0$.
\end{prop}

\begin{proof}
By integration by parts,
\[
\begin{split}
\int_{M} |\nabla^2 F|^2 \psi^2 =& -\int_{M} \langle \nabla_i \nabla_i \nabla_j F, \nabla_j f \rangle \psi^2 - 2 \int_{M} \langle \nabla_i \nabla_j F, \nabla_j F \rangle \psi \nabla_i \psi\\
=& I + II.
\end{split}
\]
From \eqref{com 5}, we have
\[
\nabla_i \nabla_i \nabla_j F = \nabla_j \nabla_i \nabla_i F + Ric_{ji} \nabla_i F.
\]
So,
\[
\begin{split}
I =& -\int_{M} \langle \nabla_j \nabla_i \nabla_i F, \nabla_j F \rangle \psi^2 - \int_{M} \langle Ric_{ji} \nabla_i F, \nabla_j F \rangle \psi^2\\
\leq & \int_{M} |\Delta F|^2 \psi^2 + 2\int_{M} \langle \Delta F, \nabla F \rangle \psi \nabla \psi + C \int_{M} |\nabla F|^2 \psi^2\\
\leq& (1+\delta) \int_{M} |\Delta F|^2 \psi^2 + C \int_{M} |\nabla F|^2 |\nabla \psi|^2 + C \int_{M} |\nabla F|^2 \psi^2\\
II \leq& \delta \int_{M} |\nabla^2 F|^2 \psi^2 + C \int_{M} |\nabla F|^2 |\nabla \psi|^2.
\end{split}
\]
Combining these together, we obtain the desired inequality.
\end{proof}

As a corollary, we have the following inequalities that are used later.

\begin{lemma} \label{Int by parts}
Let $f \in C^{\infty}(M \times [0,T],N)$, $\varphi$ be any cut-off function on $M \times [0,T]$.
Then
\begin{equation} \label{nabla sq f_t 0}
\begin{split}
\int_{M} |\nabla df_t|^2 \varphi^4 \leq& 2 \int_{M} |\Delta f_t|^2 \varphi^4  + C \int_{M} |d f_t|^2 \varphi^2 |\nabla \varphi|^2 + C \int_{M} |d f_t|^2  \varphi^4  ,
\end{split}
\end{equation}
\begin{equation} \label{nabla sq}
\int_{M} |\nabla df|^2 \varphi^2 \leq 2 \int_{M} |\Delta f|^2 \varphi^2 + C \int_{M} |df|^2 (\varphi^2 + |\nabla \varphi|^2)
\end{equation}
where $C$ depends on the Ricci curvature of $g_0$.

Also, we have
\begin{equation} \label{nabla^3 f}
\begin{split}
\int_{M} |\nabla^2 df|^2 \varphi^4 \leq& 2  \int_{M} |\Delta df|^2 \varphi^4 + C \int_{M} (|\nabla df|^2 |df|^2 + |df|^6) \varphi^4 + C\int_{M} (1 + |\nabla df|^2) \varphi^4 \\
& + C\int_{M} |\nabla df|^2 \varphi^2 |\nabla \varphi|^2
\end{split}
\end{equation}
%\begin{equation} \label{nabla^4 f}
%\begin{split}
%\int_{M} |\nabla^2 \Delta f|^2 \varphi^4 \leq& 2  \int_{M} |\Delta^2 f|^2 \varphi^4 + C \int_{M} (|\nabla df|^2 |df|^2 + |df|^6) \varphi^4 + C\int_{M} (1 + |\nabla df|^2) \varphi^4 \\
%& + C\int_{M} |\nabla df|^2 \varphi^2 |\nabla \varphi|^2
%\end{split}
%\end{equation}
%\begin{equation} \label{nabla^4 f 2}
%\begin{split}
%\int_{M} |\nabla^3 df|^2 \varphi^4 \leq& 2  \int_{M} |\Delta df|^2 \varphi^4 + C \int_{M} (|\nabla df|^2 |df|^2 + |df|^6) \varphi^4 + C\int_{M} (1 + |\nabla df|^2) \varphi^4 \\
%& + C\int_{M} |\nabla df|^2 \varphi^2 |\nabla \varphi|^2
%\end{split}
%\end{equation}
where $C$ depends on the Ricci curvature of $g_0$.

For any $\delta>0$, we have
\begin{equation} \label{df^2 df_t^2}
\begin{split}
\int_{M} |df|^2 |df_t|^2 \varphi^4 \leq& \delta \int_{M} |\Delta f_t|^2 \varphi^4 + C \int_{B_R} |f_t|^2 |df|^4 \varphi^4 + C \int_{M} |f_t|^2 |\nabla df|^2 \varphi^4 \\
&+ C \int_{M} |df_t|^2 \varphi^2 |\nabla \varphi|^2
\end{split}
\end{equation}
\begin{equation} \label{df_t nabla phi}
\int_{M} |df_t|^2 \varphi^2 |\nabla \varphi|^2 \leq \delta \int_{M} |\Delta f_t|^2 \varphi^4 + C \int_{M} |f_t|^2 |\nabla \varphi|^4 + C \int_{M} |f_t|^2 \varphi^2 |\nabla^2 \varphi|^2
\end{equation}
where $C$ depends on the Ricci curvature of $g_0$ and $\delta$.
\end{lemma}

\begin{proof}
\eqref{nabla sq f_t 0} is direct from \Cref{Int by parts general} with $F=f_t$ and $\psi = \varphi^2$.
Similarly, \eqref{nabla sq} can be obtained with $F=f$ and $\psi = \varphi$.

For \eqref{nabla^3 f}, we note that for any $X_1, X_2 \in \Gamma (TM)$,
\begin{equation} \label{comm}
\begin{split}
|(\nabla \Delta f - \Delta df)(X_1)| \leq& C (|df|^3 + |df|) |X_1|,\\
|(\nabla \Delta df - \Delta \nabla df)(X_1,X_2)| \leq& C (|df|^2 |\nabla df| + |\nabla df| + |df|^4 + |df|) |X_1| |X_2|
\end{split}
\end{equation}
where $C$ only depends on curvature tensors of $M$. (see Lamm 2004, equation 2.3 and 2.4.)
Then by integration by parts, we have
\[
\begin{split}
\int_{M} |\nabla^2 df|^2 \varphi^4 \leq & -\int_{M} \langle \nabla \Delta df, \nabla df \rangle \varphi^4 + C \int_{M} (|df|^2 |\nabla df| + |\nabla df| + |df|^4  + |df|)|\nabla df| \varphi^4\\
& + 4 \int_{M} |\nabla^2 df| |\nabla df| \varphi^3 |\nabla \varphi|\\
\leq& \int_{M} |\Delta df|^2 \varphi^4 + 4 \int_{M} |\Delta df| |\nabla df| \varphi^3 |\nabla \varphi| + C \int_{M} (|df|^2 |\nabla df|^2 + |\nabla df|^2 ) \varphi^4\\
&+ C \int_{M} ( |df|^6 + 1) \varphi^4 + \delta \int_{M} |\nabla^2 df|^2 \varphi^4 + C \int_{M} |\nabla df|^2  \varphi^2 |\nabla \varphi|^2\\
\leq& (1+\delta) \int_{M} |\Delta df|^2 \varphi^4  +  C \int_{M} (|df|^2 |\nabla df|^2 + |\nabla df|^2) \varphi^4\\
&+ C \int_{M} (|df|^6 + 1) \varphi^4 + \delta \int_{M} |\nabla^2 df|^2 \varphi^4 + C \int_{M} |\nabla df|^2  \varphi^2 |\nabla \varphi|^2.
\end{split}
\]

For \eqref{df^2 df_t^2}, by integration by parts,
\[
\begin{split}
\int_{B_R} |df|^2 |df_t|^2 \varphi^4
 =& -\int_{B_R} \langle \Delta f_t,f_t \rangle |df|^2 \varphi^4 - 2\int_{B_R} \langle df_t,f_t \rangle \langle \nabla df, df \rangle \varphi^4 - 4 \int_{B_R} \langle df_t,f_t \rangle |df|^2 \varphi^3 \nabla \varphi\\
\leq& \delta \int_{B_R} |\Delta f_t|^2 \varphi^4 + C \int_{B_R} |f_t|^2 |df|^4 \varphi^4 + \frac{1}{2} \int_{B_R} |df|^2 |df_t|^2 \varphi^4\\
& + C \int_{B_R} |f_t|^2 |\nabla df|^2 \varphi^4 + C \int_{B_R} |df_t|^2 \varphi^2 |\nabla \varphi|^2.
\end{split}
\]
Similarly, for \eqref{df_t nabla phi}, by integration by parts,
\[
\begin{split}
\int_{B_R} |d f_t|^2 \varphi^2 |\nabla \varphi|^2 =& -\int_{B_R} \langle \Delta f_t, f_t \rangle \varphi^2 |\nabla \varphi|^2- \int_{B_R} \langle d f_t, f_t \rangle 2 \varphi \nabla \varphi |\nabla \varphi|^2\\
&- \int_{B_R} \langle d f_t, f_t \rangle \varphi^2 2 \langle \nabla \varphi, \nabla^2 \varphi \rangle\\
\leq& \delta \int_{B_R} |\Delta f_t|^2 \varphi^4 + C \int_{B_R} |f_t|^2 |\nabla \varphi|^4 + \frac{1}{2} \int_{B_R} |d f_t|^2 \varphi^2 |\nabla \varphi|^2\\
&+ C \int_{B_R} |f_t|^2 \varphi^2 |\nabla^2 \varphi|.
\end{split}
\]
\end{proof}

Using commutator inequality \eqref{comm}, we can replace $|\Delta df|^2$ by $|\nabla \Delta f|^2$ in \eqref{nabla^3 f}.

Combining Sobolev embedding and using some integration by parts technique, we have more useful inequalities.

\begin{lemma} \label{Sobolev}
Let $f \in C^{\infty}(M \times [0,T])$, $\varphi$ be any cut-off function on $M \times [0,T]$.
%Also, assume $E(f) \leq E_0<\infty$.
%Also, assume that $\int_{\supp \varphi} |df|^4 \leq \ep$ for some constant $\ep$.
Then,
% for any $\delta>0$,
\begin{equation} \label{nabla sq f}
\begin{split}
\int_{M} |\nabla df|^2 |df|^2 \varphi^4 \leq& C \left( \int_{\supp \varphi} |df|^4 \right)^{\frac{1}{2}} \left( \int_{M} | \nabla^2 df|^2 \varphi^4 +  |\nabla df|^2 \varphi^2 |\nabla \varphi|^2 \right)
\end{split}
\end{equation}
\begin{equation} \label{df^8}
\begin{split}
\left( \int_{M}  |df|^8 \varphi^8 \right)^{\frac{1}{2}} \leq& C \left( \int_{\supp \varphi} |df|^4 \right)^{\frac{1}{2}} \left( \int_{M} | \nabla^2 df|^2 \varphi^4 +  |\nabla df|^2 \varphi^2 |\nabla \varphi|^2 \right)\\
&+ C \int_{M} |df|^4 \varphi^2 |\nabla \varphi|^2
\end{split}
\end{equation}
\begin{equation} \label{df^6}
\begin{split}
\int_{M}  |df|^6 \varphi^4 \leq& C \left( \int_{\supp \varphi} |df|^4 \right) \left( \int_{M} | \nabla^2 df|^2 \varphi^4 + |\nabla df|^2 \varphi^2 |\nabla \varphi|^2 \right)\\
&+ C \left( \int_{\supp \varphi} |df|^4 \right)^{\frac{1}{2}} \left( \int_{M} |df|^4 \varphi^2 |\nabla \varphi|^2 \right)
\end{split}
\end{equation}
%where $C$ depends on the Ricci curvature of $g_0$.
where $C$ is universal constant.
%\begin{equation} \label{df^4}
%\int_{M} |df|^4 \varphi^4 \leq C \int_{M} |\Delta f|^2
%\end{equation}
%Also, we have
%\begin{equation} \label{nabla^2 f}
%\begin{split}
%\int_{M} |\nabla df|^2 \varphi^4 \leq& 2 \int_{M} |\Delta f|^2 \varphi^4 + C \int_{M} |df|^6 \varphi^4 + C \int_{M} |df|^2 (\varphi^4 + \varphi^2 |\nabla \varphi|^2)
%\end{split}
%\end{equation}
\end{lemma}

\begin{proof}
By Sobolev embedding $W^{1,2}_0 \hookrightarrow L^4$ and H{\"o}lder, we have
\[
\begin{split}
\int_{M} |\nabla df|^2 |df|^2 \varphi^4 \leq& \left( \int_{\supp \varphi} |df|^4 \right)^{\frac{1}{2}} \left( \int_{M} |\nabla df|^4 \varphi^8 \right)^{\frac{1}{2}}\\
\leq& C \left( \int_{\supp \varphi} |df|^4 \right)^{\frac{1}{2}} \left( \int_{M} |\nabla^2 df|^2 \varphi^4 + |\nabla df|^2 \varphi^2 |\nabla \varphi|^2 \right).
\end{split}
\]
\[
\begin{split}
\left( \int_{M} |df|^8 \varphi^8 \right)^{\frac{1}{2}} \leq& C \left( \int_{M} |\nabla df|^2 |df|^2 \varphi^4 + |df|^4 \varphi^2 |\nabla \varphi|^2 \right)\\
\leq& C \left( \int_{\supp \varphi} |df|^4 \right)^{\frac{1}{2}} \left( \int_{M} |\nabla^2 df|^2 \varphi^4 + |\nabla df|^2 \varphi^2 |\nabla \varphi|^2 \right)\\
&+ C \left( \int_{M} |df|^4 \varphi^2 |\nabla \varphi|^2 \right)
\end{split}
\]
\[
\begin{split}
\int_{M} |df|^6 \varphi^4 \leq& \left( \int_{\supp \varphi} |df|^4 \right)^{\frac{1}{2}} \left( \int_{M} |df|^8 \varphi^8 \right)^{\frac{1}{2}}\\
\leq& C \left( \int_{\supp \varphi} |df|^4 \right) \left( \int_{M} |\nabla^2 df|^2 \varphi^4 + |\nabla df|^2 \varphi^2 |\nabla \varphi|^2 \right)\\
&+ C \left( \int_{\supp \varphi} |df|^4 \right)^{\frac{1}{2}} \left( \int_{M} |df|^4 \varphi^2 |\nabla \varphi|^2 \right).
\end{split}
\]
This completes the proof.
\end{proof}

We also have the following $W^{3,2}$ estimate.

\begin{lemma} \label{loc W32 lem}
(Local $W^{3,2}$ estimate)
Let $(f,u)$ be a smooth solution of \eqref{eq1} on $M \times [t_1,t_2]$.
%Also assume that $\int_{\supp \varphi} |df|^4 \leq \ep$ for some $\ep>0$.
Then
\begin{equation} \label{loc W32}
\begin{split}
 \int_{M} |\nabla \Delta f|^2 \varphi^4 \leq& C \int_{M} |\nabla df|^2 |df|^2 \varphi^4 + C \int_{M} |df|^6 \varphi^4 + C \int_{M} |\Delta f|^2 \varphi^2 |\nabla \varphi|^2\\
& + C \int_{M} e^{4u} |\Delta f| |f_t| \varphi^4,
\end{split}
\end{equation}
\end{lemma}

\begin{proof}
%We first show \eqref{loc W32}.
Note that
\[
|\nabla P| \leq \|\nabla P\| |df| \leq C_N |df|, \quad |\Delta P| \leq \|\nabla^2 P\| (|\Delta f| + |df|^2) \leq C_N (|\Delta f| + |df|^2).
\]

Multiply the first equation in \eqref{eq1} with $- \Delta f \varphi^4$ and integrate to get
\[
\begin{split}
-\int_{M} e^{4u} \langle f_t, \Delta f \rangle \varphi^4 =& \int_{M} \langle \Delta^2 f, \Delta f \rangle \varphi^4 - \int_{M} \langle \Delta (A(df,df)), \Delta f \rangle \varphi^4\\
&+ \int_{M} \langle \langle \Delta f, \Delta P \rangle, \Delta f \rangle \varphi^4 - 2 \int_{M} \langle \nabla \langle \Delta f, \nabla P \rangle, \Delta f \rangle \varphi^4\\
\leq & -\int_{M} |\nabla \Delta f|^2 \varphi^4 - 4 \int_{M} \langle \nabla \Delta f, \Delta f \rangle \varphi^3 \nabla \varphi\\
&+  \int_{M} \langle 2A (\nabla df, df) + DA(df,df) \cdot df, \nabla \Delta f \varphi^4 + 4 \Delta f \varphi^3 \nabla \varphi \rangle\\
&+ 3C_N \int_{M} (|\Delta f|^3  + |\Delta f|^2 |df|^2) \varphi^4 + 2C_N \int_{M} |\nabla \Delta f| |\Delta f| |df| \varphi^4\\
=& -\int_{M} |\nabla \Delta f|^2 \varphi^4 + I + II + III + 2IV.
\end{split}
\]
$I$ and $II$ are estimated by
\[
\begin{split}
I \leq& \delta \int_{M} |\nabla \Delta f|^2 \varphi^4 + C \int_{M} |\Delta f|^2 \varphi^2 |\nabla \varphi|^2.\\
II \leq& C_N \int_{M} (2|\nabla df| |df| + |df|^3) (|\nabla \Delta f| \varphi^4 + 4|\Delta f| \varphi^3 |\nabla \varphi|) \\
\leq& \frac{1}{4} \int_{M} |\nabla \Delta f|^2 \varphi^4 + 4C_N^2 \int_{M} |\nabla df|^2 |df|^2 \varphi^4 + \frac{1}{8} \int_{M} |\nabla \Delta f|^2 \varphi^4 + 2C_N^2 \int_{M} |df|^6 \varphi^4\\
&+ 4C_N \int_{M} |\nabla df|^2 |df|^2 \varphi^4  + 2C_N \int_{M} |df|^6 \varphi^4 + 6C_N \int_{M} |\Delta f|^2 \varphi^2 |\nabla \varphi|^2.
\end{split}
\]
%\[
%\begin{split}
%\leq& -(1-\delta) \int_{M} |\nabla \Delta f|^2 \varphi^4 + C \int_{M} |\Delta f|^2 \varphi^2 |\nabla \varphi|^2\\
%&+ \frac{1}{4} \int_{M} |\nabla \Delta f|^2 \varphi^4 + 4C_N^2 \int_{M} |\nabla^2 f|^2 |df|^2 \varphi^4\\
%&+ \frac{1}{8} \int_{M} |\nabla \Delta f|^2 \varphi^4 + 2C_N^2 \int_{M} |df|^6 \varphi^4\\
%&+ \delta  \int_{M} |\nabla^2 f|^2 |df|^2 \varphi^4  + \delta \int_{M} |df|^6 \varphi^4\\
%&+ C_N \int_{M} (|\Delta f|^3  + |\Delta f|^2 |df|^2) \varphi^4 + \frac{1}{4} \int_{M} |\nabla \Delta f|^2 \varphi^4 + C_N^2 |\Delta f|^2 |df|^2 \varphi^4.
%\end{split}
%\]
To estimate $III$, first note that by integration by parts,
\[
\begin{split}
3C_N \int_{M} |\Delta f|^3 \varphi^4 \leq& 9C_N  \int_{M}  |\nabla \Delta f| |df| |\Delta f| \varphi^4 + 12C_N \int_{M} |\Delta f|^2 |df| \varphi^3 |\nabla \varphi|\\
\leq& 9IV + 6 C_N \int_{M} |\Delta f|^2 |df|^2 \varphi^4 + 6 C_N \int_{M} |\Delta f|^2 \varphi^2 |\nabla \varphi|^2.
\end{split}
\]
Finally, $11 IV$ can be estimate by
\[
11 IV \leq \frac{1}{2} \int_{M} |\nabla \Delta f|^2 \varphi^4 + C(C_N) \int_{M} |\Delta f|^2 |df|^2 \varphi^4.
\]
Rearranging terms with suitable choice of $\delta$, we get desired inequality.
\end{proof}

%%%%%%%%%%%%%%%%%%%%%%%%%%%%%%%%%%%%%%%%%%%%%%%%%%%%%%%%%%%%%%%%%%%%%%%%%%%%%
\section{Local energy estimate}
\label{sec3}

In this section we work on some estimates from local energy estimate and derivative estimate.
We first derive those two estimates, and from them we get controls for several terms.
Due to the integral nature of $f_t$ term, most estimates are written in $L^p L^q$ style.

We assume $(f,u)$ be a smooth solution of \eqref{eq1} on $M \times [t_1,t_2]$ with $E(f(0)) = E_0 < \infty$.
%Many of them are similar to CHF case, but some are changed.
%Fix $B_{R}$ and $\varphi \in C_0^{\infty}(B_{R})$ be a cut-off function on $B_{R}$ such that $\varphi \equiv 1$ on $B_{R/2}$, $0 \leq \varphi \leq 1$ and $|\nabla \varphi| \leq \frac{4}{R}$.
Let $\varphi$ be a cut-off function.
% such that $|\nabla \varphi| \leq \frac{C}{R}, |\Delta \varphi|, |\nabla^2 \varphi| \leq \frac{C}{R^2}$.
Also, we assume there exists a constant $C_N$ only depending on the isometric embedding $(N,h) \hookrightarrow \R^L$ such that $\|P\|, \|\nabla P\|, \|\nabla^2 P\|, \|A\|, \|DA\|, \|D^2 A\| \leq C_N$.
%$\|R^N\|, \|P\|, \|\nabla P\|, \|\nabla^2 P\|, \|A\|, \|DA\|, \|D^2 A\| \leq C_N$ where $R^N$ is the Riemannian curvature tensor of $N$.

\begin{lemma} \label{loc E lem}
(Local energy estimate)
Let $(f,u)$ be a smooth solution of \eqref{eq1} on $M \times [t_1,t_2]$.
Then
\begin{equation} \label{loc E}
\begin{split}
\int_{M} e^{4u}|f_t|^2 \varphi^4 + \frac{d}{dt} \int_{M}  |\Delta f|^2 \varphi^4 \leq& 16 \int_{M} |\Delta f| |df_t|  \varphi^3 |\nabla \varphi| + C e^{4at} \int_{M}  |\Delta f|^2 \left( |\nabla \varphi|^4 + \varphi^2 |\Delta \varphi|^2 \right).
\end{split}
\end{equation}
\end{lemma}

\begin{proof}
From \eqref{eq1} and because of \eqref{B perp}, we have
\[
e^{4u} |f_t|^2 = -\langle \Delta^2 f, f_t \rangle.
\]
Now multiply with $\varphi^4$ and integrating, we get
\[
\begin{split}
\int_{M} e^{4u}|f_t|^2 \varphi^4 =& - \int_{M}  \langle \Delta^2 f ,f_t \rangle \varphi^4= \int_{M} \langle \nabla \Delta f, d f_t \rangle \varphi^4 + 4 \int_{M} \langle \nabla \Delta f, f_t \rangle \varphi^{3} \nabla \varphi\\
=& -\int_{M} \langle \Delta f, \Delta f_t \rangle \varphi^4  -8 \int_{M} \langle \Delta f, d f_t \rangle \varphi^3 \nabla \varphi\\
&- 12 \int_{M} \langle \Delta f, f_t \rangle \varphi^2 |\nabla \varphi|^2 - 4 \int_{M} \langle \Delta f, f_t \rangle \varphi^3 \Delta \varphi\\
%\leq& -\frac{d}{dt} \frac{1}{2} \int_{M} |\Delta f|^2 \varphi^4 + 8 \int_{M} \langle \nabla \Delta f,  f_t \rangle \varphi^3 \nabla \varphi\\
%& + 12 \int_{M} \langle \Delta f, f_t \rangle \varphi^2 |\nabla \varphi|^2 + 4 \int_{M} \langle \Delta f, f_t \rangle \varphi^3 \Delta \varphi \\
\leq& -\frac{d}{dt} \frac{1}{2} \int_{M} |\Delta f|^2 \varphi^4 + 8 \int_{M} |\Delta f| |df_t| \varphi^3 |\nabla \varphi| \\
& + \frac{1}{2} \int_{M} e^{4u} |f_t|^2 \varphi^4 + C \int_{M} e^{-4u}|\Delta f|^2 \left( |\nabla \varphi|^4 + \varphi^2 |\Delta \varphi|^2 \right)
\end{split}
\]
which completes the proof.
\end{proof}

Note that we can also obtain another version of local energy estimate as below:
\begin{equation} \label{loc E 2}
\begin{split}
\int_{M} e^{4u}|f_t|^2 \varphi^4 + \frac{d}{dt} \int_{M}  |\Delta f|^2 \varphi^4 \leq& C e^{4at} \int_{M} |\nabla \Delta f|^2  \varphi^2 |\nabla \varphi|^2 + C e^{4at} \int_{M}  |\Delta f|^2 \left( |\nabla \varphi|^2 + |\Delta \varphi| \right)^2.
\end{split}
\end{equation}
Its proof is similar to above, where we apply another integration by parts to the term $\int_{M} \langle \Delta f, df_t \rangle \varphi^3 \nabla \varphi$ to remove spatial derivative of the factor $df_t$.

%Next, we show the derivative estimate for $p=0$.
%Later we extend it to the general version for $p>0$.
Next, we show the derivative estimate.
This can be considered as $p=0$ version of more general estimates, like in \cite{P22} and \cite{P25}.
But in biharmonic map flow case, we can only get the version below and not a general version with higher $p$.
For example, if we try to estimate the time derivative of $\int e^{4u} |f_t|^{p+2}$, on RHS, a good term (with negative sign) is $\int |\Delta f_t|^2 |f_t|^p$, while one of the trouble terms that we ultimately need to control is $p \int |df_t|^4 |f_t|^{p-2}$.
And we do not have local control of $p \int |df_t|^4 |f_t|^{p-2}$ in terms of $\int |\Delta f_t|^2 |f_t|^p$ except $p=0$.

\begin{prop} \label{Der p=0}
%(Derivative estimate for $p=0$)
(Derivative estimate)

Let $(f,u)$ be a smooth solution of \eqref{eq1} on $M \times [t_1,t_2]$.
Then
\begin{equation} \label{Der p=0 eq}
\begin{split}
\frac{d}{dt}  \int_{B_R} &  e^{4u}|f_t|^2 \varphi^4  + \int_{B_R} |\Delta f_t|^2 \varphi^4 + C_b \int_{B_R} |f_t|^2 (|\nabla df|^2 + |df|^4) \varphi^4\\
 \leq& 4a \int_{B_R} e^{4u}|f_t|^2 \varphi^4 %
 %+ C \int_{B_R} |\nabla f_t|^2 \varphi^2 |\nabla \varphi|^2 
 + C \int_{B_R} |f_t|^2 (\varphi^4 + |\nabla \varphi|^4 + \varphi^2 |\Delta \varphi|^2) .
\end{split}
\end{equation}
Here $C$ only depends on the Ricci curvature of $g_0$ and $C_N$ and $C_b$ depends on $b$, the Ricci curvature of $g_0$, and $C_N$.
And $C_b > 0$ for sufficiently large $b$.
%$ = 4b - C$ for some constant $C$ only depends on Ricci curvature of $g_0$ and $C_N$.
\end{prop}

\begin{proof}
From the first equation in \eqref{eq1}, take time derivative and get
\begin{equation} \label{tau_t}
\begin{split}
(e^{4u}f_t)_t =& -\left[ \Delta^2 f \right]_t + \left[ \Delta (A(df,df)) \right]_t - \left[ \langle \Delta f, \Delta P \rangle \right]_t + 2 \left[ \nabla \langle \Delta f, \nabla P \rangle \right]_t.
\end{split}
\end{equation}
Take inner product with $f_t \varphi^4$ and integrate to get
\[
\begin{split}
\int_{B_R} \langle (e^{4u} f_t)_t, f_t \rangle  \varphi^4 %=& -\int_{B_R} \langle \left[\Delta^2 f \right]_t, f_t  \rangle \varphi^4 + \int_{B_R} \langle \left[ \Delta (A(df,df)) \right]_t , f_t \rangle \varphi^4 - \int_{B_R} \langle \left[ \langle \Delta f, \Delta P \rangle \right]_t , f_t \rangle \varphi^4\\
%&+ \int_{B_R} \langle 2 \left[ \nabla \langle \Delta f, \nabla P \rangle \right]_t, f_t \rangle \varphi^4\\
&= I + II + III + IV.
\end{split}
\]
Each term can be estimate as follows. 
%Here $\delta_1, \delta_2, \cdots$ are (different) small positive numbers to be determined later.

For $I$, we have
\[
\begin{split}
I =& \int_{B_R} \langle  \nabla_i \Delta f_t, \nabla_i f_t  \varphi^4 +  4 f_t  \varphi^3 \nabla_i \varphi \rangle \\
=& -\int_{B_R} |\Delta f_t|^2 \varphi^4 -  \int_{B_R}  \langle \Delta f_t, 8 \nabla_i f_t  \varphi^3 \nabla_i \varphi + 12 f_t \varphi^2 |\nabla \varphi|^2 + 4 f_t  \varphi^3 \Delta \varphi \rangle\\
\leq& -(1-\delta) \int_{B_R} |\Delta f_t|^2 \varphi^4 + C \int_{B_R} |f_t|^2 (|\nabla \varphi|^4 + \varphi^2 |\Delta \varphi|^2) + |df_t|^2 \varphi^2 |\nabla \varphi|^2.
\end{split}
\]

For $II$, note that
\[
\begin{split}
\left[ \Delta (A(df,df)) \right]_t  =& \Delta \left( 2 A(df_t,df) + DA(df,df) \cdot f_t \right)\\
=& \nabla_i \left(  2 A(\nabla_i df_t, df) + 2 A(df_t, \nabla_i df) + 2DA(df_t,df) \cdot \nabla_i f \right)\\
& +\nabla_i \left( 2DA(\nabla_i df, df) \cdot f_t+ D^2A(df,df) \cdot f_t \cdot \nabla_i f + DA(df,df) \cdot \nabla_i f_t \right).
\end{split}
\]
Hence, by integration by parts, we get
\[
\begin{split}
II =& -\int_{B_R} \langle 2A(\nabla_i df_t,df) + 2A(df_t,\nabla_i df) + 2DA(df_t,df) \cdot \nabla_i f, \nabla_i f_t \varphi^4 + f_t \varphi^3 \nabla_i \varphi \rangle\\
&- \int_{B_R} \langle 2DA(\nabla_i df, df) \cdot f_t + D^2A(df,df) \cdot f_t \cdot \nabla_i f + DA(df,df) \cdot \nabla_i f_t,  \nabla_i f_t \varphi^4 + f_t \varphi^3 \nabla_i \varphi \rangle.
\end{split}
\]
Here we would like to avoid the term $\int |df_t|^2 |\nabla df| \varphi^4$.
So, apply integration by parts again for that term.
\[
\begin{split}
-\int_{B_R} \langle &2A(df_t,\nabla_i df), \nabla_i f_t \varphi^4 \rangle\\
 =& \int_{B_R} \langle 2A(\nabla_i df_t, df), \nabla_i f_t \varphi^4 \rangle + \int_{B_R} \langle 2DA(df_t,df) \cdot \nabla_i f, \nabla_i f_t \varphi^4 \rangle\\
&+ \int_{B_R} \langle 2A(df_t,df), \Delta f_t \varphi^4 \rangle + 4\int_{B_R} \langle 2A(df_t,df), \nabla_i f_t \varphi^3 \nabla_i \varphi \rangle\\
\leq& C \int_{B_R} |\nabla df_t| |df_t| |df| \varphi^4 + C \int_{B_R} |df_t|^2 |df|^2 \varphi^4 + C \int_{B_R} |\Delta f_t| |df_t| |df| \varphi^4\\
& + C \int_{B_R} |df_t|^2 |df| \varphi^3 |\nabla \varphi|.
\end{split}
\]
Considering this inequality with other terms, we obtain
\[
\begin{split}
II \leq &C \int_{B_R}  ( |\nabla df_t| |df| + |df_t| |df|^2 + |\nabla df| |df| |f_t| + |df|^3 |f_t| ) \cdot ( |df_t| \varphi^4 + |f_t| \varphi^3 |\nabla \varphi| )\\
&+ C \int_{B_R} |\nabla df_t| |df_t| |df| \varphi^4 + C \int_{B_R} |df_t|^2 |df|^2 \varphi^4 + C \int_{B_R} |df_t|^2 |df| \varphi^3 |\nabla \varphi|\\
\leq& \delta \int_{B_R} |\nabla df_t|^2 \varphi^4 + C \int_{B_R} \left( |df_t|^2 |df|^2  + |f_t|^2 |\nabla df|^2  + |f_t|^2 |df|^4 \right) \varphi^4\\
&+ C \int_{B_R} \left(  |f_t|^2 |df|^2 + |df_t|^2 \right) \varphi^2 |\nabla \varphi|^2\\
\leq& \delta \int_{B_R} |\nabla df_t|^2 \varphi^4 + C \int_{B_R} |df|^2 |df_t|^2 \varphi^4\\
&+ C \int_{B_R} |f_t|^2 (|df|^4 + |\nabla df|^2) \varphi^4 + C \int_{B_R} \left(  |f_t|^2 |df|^2 + |df_t|^2 \right) \varphi^2 |\nabla \varphi|^2.
\end{split}
\]

To analyze $III$, note that
\[
\begin{split}
\langle \Delta f, \Delta P \rangle =& A(\Delta f, \Delta f) + DA(df,\Delta f) \cdot df.
%\langle \nabla \Delta f, \nabla P \rangle =& \nabla \langle \Delta f, \nabla P \rangle - \langle \Delta f, \Delta P \rangle.
\end{split}
\]
So, we have
\[
\begin{split}
III =& -\int_{B_R} \langle 2A(\Delta f_t, \Delta f), f_t \rangle \varphi^4 - \int_{B_R} \langle DA(\Delta f, \Delta f) \cdot f_t, f_t \rangle \varphi^4\\
& - \int_{B_R} \langle DA(df_t, \Delta f) \cdot df, f_t \rangle \varphi^4 - \int_{B_R} \langle DA(df, \Delta f_t) \cdot df, f_t \rangle \varphi^4\\
&- \int_{B_R} \langle DA(df, \Delta f) \cdot df_t, f_t \rangle \varphi^4 - \int_{B_R} \langle D^2 A(df, \Delta f) \cdot df \cdot f_t, f_t \rangle \varphi^4\\
\leq& \delta \int_{B_R} |\Delta f_t|^2 \varphi^4 + C \int_{B_R} |df|^2 | df_t|^2 \varphi^4 +  C \int_{B_R} |f_t|^2 (|df|^4 + |\Delta f|^2) \varphi^4.
\end{split}
\]
For $IV$, first note that
\[
\begin{split}
\langle \Delta f, \nabla P \rangle =& A(\Delta f, df)\\
\left[ \nabla \langle \Delta f, \nabla P \rangle \right]_t %=& \nabla \left[ \langle \Delta f, \nabla P \rangle \right]_t + {}^N\!\! R(f_t, df) A(\Delta f, df) \\
 =& \nabla \left( A(\Delta f_t, df) + A(\Delta f, df_t) + DA(\Delta f, df) \cdot f_t \right) %+ {}^N\!\! R(f_t, df) A(\Delta f, df)
\end{split}
\]
and apply integration by parts to get
\[
\begin{split}
IV =& -2 \int_{B_R} \langle A(\Delta f_t, df) + A(\Delta f, df_t) + DA(\Delta f, df) \cdot f_t, df_t \varphi^4 + 4 f_t \varphi^3 \nabla \varphi \rangle.
\end{split}
\]
Again, we would like to avoid the term $\int |\Delta f| |df_t|^2 \varphi^4$.
So, apply integration by parts again to get
\[
\begin{split}
-2 \int_{B_R} \langle &A(\Delta f, df_t), df_t \varphi^4 \rangle\\
 =& 2\int_{B_R} \langle A(\nabla_i f, \nabla_i df_t),df_t \varphi^4 \rangle + 2 \int_{B_R} \langle DA(\nabla_i f, df_t) \cdot \nabla_i f, df_t \varphi^4 \rangle\\
&+ 2 \int_{B_R} \langle A(\nabla_i f, df_t), \nabla_i df_t \varphi^4 \rangle + 8 \int_{B_R} \langle A(\nabla_i f, df_t), df_t \varphi^3 \nabla_i \varphi \rangle\\
\leq& C \int_{B_R} |\nabla df_t| |df_t| |df| \varphi^4 + C \int_{B_R} |df_t|^2 |df|^2 \varphi^4 + C \int_{B_R} |df_t|^2 |df| \varphi^3 |\nabla \varphi|.
\end{split}
\]
Considering this estimate with other terms, we obtain
\[
\begin{split}
IV %=& 2\int_{B_R} \langle \nabla (\langle \Delta f, \nabla P \rangle)_t, f_t \rangle \varphi^4\\
% 2 \int_{B_R} \langle A(df, \nabla \Delta f_t),f_t \rangle \varphi^4 + 2 \int_{B_R} \langle A(df_t,\nabla \Delta f),f_t \rangle \varphi^4\\
%&+ 2 \int_{B_R} \langle DA(df, \nabla \Delta f) \cdot f_t, f_t \rangle \varphi^4\\
%=& -2 \int_{B_R} \langle A(\Delta f_t, df) + A(\Delta f, df_t) + DA(\Delta f, df) \cdot f_t, df_t \varphi^4 + 4 f_t \varphi^3 \nabla \varphi \rangle\\
%&+ 2 \int_{B_R} \langle {}^N\!\! R(f_t, df) A(\Delta f, df), f_t \rangle \varphi^4\\
\leq& \delta \int_{B_R} |\Delta f_t|^2 \varphi^4 + C  \int_{B_R} |df|^2 |d f_t|^2 \varphi^4 \\
&+ C \int_{B_R} |f_t|^2 (|df|^4 + |\Delta f|^2) \varphi^4 + C \int_{B_R} (|f_t|^2 |df|^2 + |df_t|^2) \varphi^2 |\nabla \varphi|^2.
\end{split}
\]

Note that for $\int |df|^2 |df_t|^2 \varphi^2$ and $\int |df_t|^2 \varphi^2 |\nabla \varphi|^2$, we use \eqref{df^2 df_t^2} and \eqref{df_t nabla phi}.

Combining all together with $\delta$ small enough, % and by \eqref{nabla sq 0},
we get
\[
\begin{split}
RHS \leq& -\frac{1}{2} \int_{B_R} |\Delta f_t|^2 \varphi^4+ \frac{1}{2}C_1 \int_{B_R} |f_t|^2 ( |df|^4 + |\nabla df|^2) \varphi^4\\
& + C \int_{B_R} |f_t|^2 (\varphi^4 + |\nabla \varphi|^4 + \varphi^2 |\nabla^2 \varphi|^2)
\end{split}
\]
where $C, C_1$ only depends on the Ricci curvature of $g_0$ and $C_N$.

On the other hand, left-hand side becomes
\[
\begin{split}
\int_{B_R} \langle (e^{4u} f_t)_t, f_t \rangle \varphi^4 =&  \int_{B_R} e^{4u} \langle f_{tt}, f_t \rangle \varphi^4 + 4 \int_{B_R} e^{4u} u_t |f_t|^2 \varphi^4\\
=& \frac{d}{dt} \frac{1}{2} \int_{B_R} e^{4u} |f_t|^2 \varphi^4 + 2b \int_{B_R} |f_t|^2 (|\nabla df|^2 + |df|^4) \varphi^4 - 2a \int_{B_R} e^{4u} |f_t|^2 \varphi^4.
\end{split}
\]
Hence we have
\[
\begin{split}
\frac{d}{dt}  \int_{B_R} & e^{4u}|f_t|^2 \varphi^4 +  \int_{B_R} |\Delta f_t|^2 \varphi^4 + (4b - C_1) \int_{B_R} |f_t|^2 (|\nabla df|^2 + |df|^4) \varphi^4\\
 \leq& 4a \int_{B_R} e^{4u}|f_t|^2 \varphi^4 + C \int_{B_R} |f_t|^2 (\varphi^4 + |\nabla \varphi|^4 + \varphi^2 |\nabla^2 \varphi|^2).
\end{split}
\]
The proof is complete with $C_b = 4b - C_1$.
\end{proof}

Form now on, we assume $b$ large enough so that
\begin{equation}
C_b = 4b-C_1 > 0.
\end{equation}

Then \Cref{Der p=0} implies, using $e^{-4u} \leq e^{4at}$,
\begin{equation} \label{Der p=0 cor}
\begin{split}
\frac{d}{dt} \int_{B_R} e^{4u}|f_t|^2 \varphi^4 +  \int_{B_R} |\Delta f_t|^2 \varphi^4 \leq&4a \int_{B_R} e^{4u}|f_t|^2 \varphi^4 +  C e^{4at} (1 + \frac{1}{R^4}) \int_{B_R} e^{4u} |f_t|^2.
\end{split}
\end{equation}
This leads to the following finer version of local energy estimate.

\begin{prop} \label{loc E finer}
(Local energy estimate - finer version)

Under the same assumption of \Cref{Der p=0} with $C_b>0$,
\begin{equation}
\begin{split}
&\int_{B_R} e^{4u}|f_t|^2 \varphi^4(t)+ 4a \int_{B_R} |\Delta f|^2 \varphi^4(t)\\
\leq& K_1(t_1) +C (t_2-t_1) e^{4at_2} (1 + \frac{1}{R^4})  \left( K_1(t_1) + C e^{4at_2} (1 + \frac{1}{R^4}) (1+t_2) E(0) \right)
\end{split}
\end{equation}
where
\begin{equation}
K_1(t_1) :=  \int_{B_{2R}} e^{4u}|f_t|^2 (t_1) + 4a \int_{B_{2R}} |\Delta f|^2  (t_1).
\end{equation}
\end{prop}

\begin{proof}
Let $\psi$ be a cut-off function on $B_{2R}$ with $\psi \equiv 1$ on $B_R$.
From \Cref{loc E lem}, we have
\[
\int_{B_{2R}} e^{4u} |f_t|^2 \psi^4 + \frac{d}{dt} \int_{B_{2R}} |\Delta f|^2 \psi^4 \leq C \int_{B_{2R}} |\Delta f| |df_t| \psi^3 |\nabla \psi| + C e^{4at}\frac{1}{R^4}  \int_{B_{2R}} |\Delta f|^2.
\]
The first term in RHS can be estimated by
\[
\int_{B_{2R}} |\Delta f| |df_t| \psi^3 |\nabla \psi| \leq \int_{B_{2R}} |\Delta f|^2 \psi^2 |\nabla \psi|^2 + C \int_{B_{2R}} |df_t|^2 \psi^4
\]
where the second term is, by integration by parts and by \eqref{df_t nabla phi},
\[
\begin{split}
\int_{B_{2R}} |df_t|^2 \psi^4 =& -\int_{B_{2R}} \langle \Delta f_t, f_t \rangle \psi^4 - 4 \int_{B_{2R}} \langle d f_t, f_t \rangle \psi^3 \nabla \psi\\
\leq& \delta \int_{B_{2R}} |\Delta f_t|^2 \psi^4 + C \int_{B_{2R}} |f_t|^2 \psi^4 + C \int_{B_{2R}} |df_t|^2 \psi^2 |\nabla \psi|^2\\
\leq& 2\delta \int_{B_{2R}} |\Delta f_t|^2 \psi^4 + C \int_{B_{2R}} |f_t|^2 ( \psi^4 + |\nabla \psi|^4 + \psi^2 |\nabla^2 \psi|^2).
\end{split}
\]

Combining this with \eqref{Der p=0 cor}, with $\delta$ small enough, we obtain that
\begin{equation} \label{loc gronwall}
\frac{d}{dt} \left( \int_{B_{2R}} e^{4u}|f_t|^2 \psi^4 + 4a \int_{B_{2R}} |\Delta f|^2 \psi^4 \right) \leq C e^{4at_2} (1 + \frac{1}{R^4}) \left( \int_{B_{2R}} e^{4u}|f_t|^2 + 4a \int_{B_{2R}} |\Delta f|^2 \right).
\end{equation}
Integrating over $[t_1,t]$, and using \Cref{E dec}, we get
\[
\begin{split}
&\int_{B_R} e^{4u}|f_t|^2 (t)+ 4a \int_{B_R} |\Delta f|^2 (t) \\
 \leq& \int_{B_{2R}} e^{4u}|f_t|^2 \psi^4(t_1) + 4a \int_{B_{2R}} |\Delta f|^2 \psi^4 (t_1) + C e^{4at_2} (1 + \frac{1}{R^4})  \left( E(t_1) - E(t) + (t-t_1) E(0) \right)\\
 \leq& K_1(t_1) + C e^{4at_2} (1 + \frac{1}{R^4}) (1+t_2) E(0)
\end{split}
\]
where we denote
\[
K_1(t_1) :=  \int_{B_{2R}} e^{4u}|f_t|^2 (t_1) + 4a \int_{B_{2R}} |\Delta f|^2  (t_1).
\]
Integrating \eqref{loc gronwall} again on $B_R$ with cut-off function $\varphi$ and using above bound, we get the conclusion.
\end{proof}

For simplicity, denote $K_2$ depending on $K_1(t_1), t_2, t_2-t_1, R, E(0)$ by
\[
K_2 := K_1(t_1) +C (t_2-t_1) e^{4at_2} (1 + \frac{1}{R^4})  \left( K_1(t_1) + C e^{4at_2} (1 + \frac{1}{R^4}) (1+t_2) E(0) \right).
\]
Then \Cref{loc E finer} implies that,
\[
\int_{B_{R/2}} e^{4u}|f_t|^2 \leq K_2, \int_{B_{R/2}} |\Delta f|^2 \leq \frac{K_2}{4a}.
\]
%As we have a good control of $\int_{B_R} e^{4u}|f_t|^2 \varphi^4 \leq K_2$, we can derive further estimates for higher derivatives.

\begin{remark}
This is a key moment that the conformal heat flow method works better than usual heat flow of biharmonic maps.
In the usual biharmonic map flow, we {\bf cannot} control the local energy in terms of its initial data and the time difference in such a uniform way.
But in our case, if we choose $R$ small enough so that $K_1(t_1)$ is small, and choose $t_2-t_1$ small enough, then we can make $K_2$ as small as we want.

More precisely, we have the following corollary. 
\end{remark}

\begin{cor} \label{loc E cor}
Let $(f,u)$ be a smooth solution of \eqref{eq1} on $M \times [t_1,t_2]$.
For any $\ep>0$, there exists $R = R(\ep)$ and $T \in [t_1,t_2]$ such that
\begin{equation} \label{loc E est finer}
\sup_{t_1 \leq s \leq T} \int_{B_R} e^{4u}|f_t|^2(s), \sup_{t_1 \leq s \leq T} \int_{B_R} |\Delta f|^2 (s) \leq \ep.
\end{equation}
\end{cor}

\begin{proof}
For the first term, it is enough to show that we can make $K_2$ small enough.
First, we choose $R$ small enough so that $K_1(t_1)$ on $B_{4R}$ is less than $\ep$.
Now choose $T$ with $T-t_1$ small enough so that
\[
C (T-t_1) e^{4at_2} (1 + \frac{1}{R^4}) \left( K_1(t_1) + C e^{4at_2}(1 + \frac{1}{R^4}) (1+t_2) E(0) \right) \leq \ep.
\]
By \Cref{loc E finer}, this completes the proof.
\end{proof}

\begin{remark}
The choice of $T$ in the above depends on $t_1$ clearly, and also depends on $t_2$ by construction.
But the dependency on $t_2$ is in a way that smaller $t_2$ will allow to choose larger $T$.
%Since $t \leq t_2$, we can have an optimal possible $t$, which can be regarded as a uniform choice of such $t \leq t_2$ for any time period.
%Or, we can choose $t_2$ small enough so that $t$ above equals to $t_2$ in this case.
So, by choosing $t_2$ small enough, we can assume that $T=t_2$ in \eqref{loc E est finer}.
\end{remark}

In particular, good control for $\int e^{4u}|f_t|^2$ will boost higher estimates.

\begin{cor} \label{delta f_t lem}
Let $(f,u)$ be a smooth solution of \eqref{eq1} on $M \times [t_1,t_2]$ and $R = R(\ep)$ be chosen as in \Cref{loc E cor}.
Then
\begin{equation} \label{delta f_t}
\int_{t_1}^{t_2} \int_{B_R} |\Delta f_t|^2 \varphi^4, \int_{t_1}^{t_2} \int_{B_R} |f_t|^2 (|\nabla df|^2 + |df|^4) \varphi^4 \leq C_3 (1 + \frac{1}{R^4})\ep
\end{equation}
for some constant $C_3 = C_3(t_2)$.
\end{cor}

\begin{proof}
We integrate \eqref{Der p=0 eq} from $t_1$ to $t$ to get
\[
\begin{split}
&\int_{t_1}^{t} \int_{B_R} |\Delta f_t|^2 \varphi^4 + C_b \int_{t_1}^{t} \int_{B_R} |f_t|^2 (|\nabla df|^2 + |df|^4) \varphi^4\\
\leq& \int_{B_R} e^{4u} |f_t|^2 \varphi^4 (t_1) - \int_{B_R} e^{4u}|f_t|^2 \varphi^4 (t)\\
&+4a \int_{t_1}^{t} \int_{B_R} e^{4u}|f_t|^2 \varphi^4 + C e^{4at_2} (1+  \frac{1}{R^4}) \int_{t_1}^{t} \int_{B_R} e^{4u}|f_t|^2.
\end{split}
\]
Now from \eqref{loc E est finer}, we get the conclusion.
\end{proof}

From this, we can estimate the followings as well.
\begin{cor}
Let $(f,u)$ be a smooth solution of \eqref{eq1} on $M \times [t_1,t_2]$ and $R = R(\ep)$ be chosen as in \Cref{loc E cor}.
Then
\begin{equation} \label{df_t^2 df^2}
\begin{split}
\int_{t_1}^{t_2} \int_{B_R} |df_t|^2 |df|^2 \varphi^4, \int_{t_1}^{t_2} \int_{B_R} |df_t|^2 \varphi^4 \leq& C_4 (1 + \frac{1}{R^4}) \ep\\
\int_{t_1}^{t_2} \int_{B_R} |\nabla df_t|^2 \varphi^4 \leq& C_4 (1 + \frac{1}{R^6}) \ep
\end{split}
\end{equation}
for some constant $C_4 = C_4(t_2, t_2-t_1)$.
\end{cor}

\begin{proof}
We first show the estimate for $\int |df_t|^2 |df|^2 \varphi^4$.
By integration by parts,
\[
\begin{split}
\int_{B_R} |df_t|^2 |df|^2 \varphi^4 =& -\int_{B_R} \langle \Delta f_t, f_t \rangle |df|^2 \varphi^4 - 2 \int_{B_R} \langle df_t, f_t \rangle \langle \nabla df, df \rangle \varphi^4\\
& - 4 \int_{B_R} \langle df_t, f_t \rangle |df|^2 \varphi^3 \nabla \varphi\\
\leq& \frac{1}{2} \int_{B_R} |\Delta f_t|^2 \varphi^4 + \frac{1}{2} \int_{B_R} |f_t|^2 |df|^4 \varphi^4 + \frac{1}{2} \int_{B_R} |df_t|^2 |df|^2 \varphi^4\\
& + C \int_{B_R} |f_t|^2 |\nabla df|^2 \varphi^4 + C \int_{B_R} |f_t|^2 |df|^2 \varphi^2 |\nabla \varphi|^2.
\end{split}
\]
And the last term can be estimated by
\[
\int_{B_R} |f_t|^2 |df|^2 \varphi^2 |\nabla \varphi|^2 \leq \int_{B_R} |f_t|^2 |df|^4 \varphi^4 + C e^{4at_2}\frac{1}{R^4} \int_{B_R} e^{4u} |f_t|^2.
\]
Therefore, we obtain
\[
\int_{t_1}^{t_2} \int_{B_R} |df_t|^2 |df|^2 \varphi^4 \leq C \ep + C e^{4at_2} \frac{1}{R^4} (t_2-t_1) \ep \leq C_4 (1 + \frac{1}{R^4}) \ep
\]
for suitably chosen $C_4$ which only depends on $t_2$ and $t_2-t_1$.
Those for $\iint |df_t|^2 \varphi^4$ is similar.

Finally, for $\iint |\nabla df_t|^2 \varphi^4$, by \eqref{nabla sq f_t 0},
\[
\begin{split}
\int_{t_1}^{t_2} \int_{B_R} |\nabla df_t|^2 \varphi^4 \leq& 2 \int_{t_1}^{t_2}\int_{B_R} |\Delta f_t|^2 \varphi^4 + C (1 + \frac{1}{R^2}) \int_{t_1}^{t_2} \int_{B_R} |df_t|^2\\
 \leq& C_4 (1 + \frac{1}{R^6}) \ep.
\end{split}
\]
\end{proof}

%We may expect that we can obtain similar estimate for $\iint |f_t|^4$ and $\iint |df_t|^2 |f_t|^2$ as above.
%But due to the different power in Sobolev embedding, we actually have the following version of the estimate.
We can apply Sobolev embedding to extend similar types of estimate.
But its result may have different powers in spatial integral and time integral.

\begin{lemma} \label{f_t^4 with ep}
Let $(f,u)$ be a smooth solution of \eqref{eq1} on $M \times [t_1,t_2]$ and $R = R(\ep)$ be chosen as in \Cref{loc E cor}.
Then
\begin{equation}
\begin{split}
\int_{t_1}^{t_2} \left( \int_{B_R} |f_t|^4 \varphi^4 \right)^{\frac{1}{2}} \leq& C_5 (1 + \frac{1}{R^4}) \ep \\
\int_{t_1}^{t_2} \left( \int_{B_R} |df_t|^4 \varphi^4 \right)^{\frac{1}{2}} , \int_{t_1}^{t_2} \left( \int_{B_R} |df_t|^2 |f_t|^2 \varphi^4 \right)^{\frac{1}{2}} \leq& C_5 (1 + \frac{1}{R^6}) \ep
\end{split}
\end{equation}
for some constant $C_5 = C_5(t_2,t_2-t_1)$.
\end{lemma}

\begin{proof}
By Sobolev embedding, we have
\[
\begin{split}
\left( \int_{B_R} |f_t|^4 \varphi^4 \right)^{\frac{1}{2}} \leq& C \int_{B_R} |df_t|^2 \varphi^2 + |f_t|^2 |\nabla \varphi|^2\\
\int_{t_1}^{t_2} \left( \int_{B_R} |f_t|^4 \varphi^4 \right)^{\frac{1}{2}} \leq& C_4 (1 + \frac{1}{R^4}) \ep + C e^{4at_2} \frac{1}{R^2} \ep\\
\leq & C_5 (1 + \frac{1}{R^4}) \ep.
\end{split}
\]
Also, we have
\[
\begin{split}
\left( \int_{B_R} |df_t|^4 \varphi^4 \right)^{\frac{1}{2}} \leq & C \left( \int_{B_R} |\nabla df_t|^2 \varphi^2 + |df_t|^2 |\nabla \varphi|^2 \right)
\end{split}
\]
and integrating over $[t_1,t_2]$ gives the conclusion.
Finally,
\[
\begin{split}
\left( \int_{B_R} |df_t|^2 |f_t|^2 \varphi^4 \right)^{\frac{1}{2}} \leq & \left( \int_{B_R} |df_t|^4 \varphi^4 \right)^{\frac{1}{4}} \left( \int_{B_R} |f_t|^4 \varphi^4 \right)^{\frac{1}{4}}\\
%\leq& C \left( \int_{B_R} |\nabla df_t|^2 \varphi^2 + |df_t|^2 |\nabla \varphi|^2 \right)^{\frac{1}{2}} \left( \int_{B_R} |f_t|^4 \varphi^4 \right)^{\frac{1}{4}}\\
\leq& \frac{1}{2} \left( \int_{B_R} | df_t|^4 \varphi^4\right)^{\frac{1}{2}} + \frac{1}{2} \left( \int_{B_R} |f_t|^4 \varphi^4 \right)^{\frac{1}{2}}.
\end{split}
\]
Integrating over $[t_1,t_2]$ and using above result, we complete the proof.
\end{proof}

We can improve the estimate of the form $\int \left( \int_{B_R} |f_t|^p \varphi^4 \right)^q$ further.
The key ingredients are above results and $\int |f_t|^2 \leq \ep$.

\begin{lemma} \label{f_t^4 with ep 2}
Let $(f,u)$ be a smooth solution of \eqref{eq1} on $M \times [t_1,t_2]$ and $R = R(\ep)$ be chosen as in \Cref{loc E cor}.
Then
\begin{equation}
\begin{split}
\int_{t_1}^{t_2} \int_{B_R} |f_t|^3 \varphi^4 \leq& C_6 (1 + \frac{1}{R^4}) \ep^{\frac{3}{2}} \\
\int_{t_1}^{t_2} \left( \int_{B_R} |f_t|^6 \varphi^4 \right)^{\frac{1}{2}} \leq& C_6 (1 + \frac{1}{R^6}) \ep^{\frac{3}{2}}\\
\int_{t_1}^{t_2} \int_{B_R} |f_t|^4 \varphi^4 \leq& C_6 (1 + \frac{1}{R^6}) \ep^2\\
\int_{t_1}^{t_2} \left( \int_{B_R} |f_t|^3 \varphi^4 \right)^2 \leq& C_6 (1 + \frac{1}{R^6}) \ep^3
\end{split}
\end{equation}
for some constant $C_6 = C_6(t_2,t_2-t_1)$.
\end{lemma}

\begin{proof}
First note that
\[
\begin{split}
\int_{t_1}^{t_2} \int_{B_R}|f_t|^3 \varphi^4 \leq& \int_{t_1}^{t_2} \left(\int_{B_R} |f_t|^2 \varphi^4 \right)^{\frac{1}{2}} \left( \int_{B_R} |f_t|^4 \varphi^4 \right)^{\frac{1}{2}} \\
\leq& C e^{2at_2} \ep^{\frac{1}{2}} C_5 (1 + \frac{1}{R^4}) \ep \leq C_6 (1 + \frac{1}{R^4}) \ep^{\frac{3}{2}}.
\end{split}
\]
Now this implies, by Sobolev embedding $W^{1,2}_0 \hookrightarrow L^4$,
\[
\begin{split}
\int_{t_1}^{t_2} \left( \int_{B_R} |f_t|^6 \varphi^4 \right)^{\frac{1}{2}} \leq& C \int_{t_1}^{t_2} \int_{B_R} |df_t|^2 |f_t| \varphi^2 + |f_t|^3 |\nabla \varphi|^2\\
\leq& C \int_{t_1}^{t_2} \left( \int_{B_R} |f_t|^2 \right)^{\frac{1}{2}} \left( \int_{B_R} |df_t|^4 \varphi^4 \right)^{\frac{1}{2}} + C_6 (1 + \frac{1}{R^6}) \ep^{\frac{3}{2}}\\
\leq& C_6 (1 + \frac{1}{R^6}) \ep^{\frac{3}{2}}.
\end{split}
\]
Finally,
\[
\begin{split}
\int_{t_1}^{t_2} \int_{B_R} |f_t|^4 \varphi^4 \leq& \int_{t_1}^{t_2} \left( \int_{B_R} |f_t|^2 \varphi^4 \right)^{\frac{1}{2}} \left( \int_{B_R} |f_t|^6 \varphi^4 \right)^{\frac{1}{2}} \leq C_6 (1 + \frac{1}{R^6}) \ep^2
\end{split}
\]
which also implies that
\[
\begin{split}
\int_{t_1}^{t_2} \left( \int_{B_R} |f_t|^3 \varphi^4 \right)^2 \leq& C \int_{t_1}^{t_2} \left( \int_{B_R} |f_t|^2 \varphi^4 \right) \left( \int_{B_R} |f_t|^4 \varphi^4 \right) \leq C_6 (1 + \frac{1}{R^6}) \ep^3.
\end{split}
\]
\end{proof}

%Using above result, we can also obtain estimates of the form $\int \left( \int_{B_R} |df_t|^2 |f_t|^p \varphi^4 \right)^q$.
%Note that the arguments are like bootstrapping style, hence can increase $p$ as much as we want.
Using above result, we can improve the estimate in \Cref{f_t^4 with ep} as follows.

\begin{lemma} \label{f_t^4 with ep 3}
Let $(f,u)$ be a smooth solution of \eqref{eq1} on $M \times [t_1,t_2]$ and $R = R(\ep)$ be chosen as in \Cref{loc E cor}.
Then
\begin{equation}
\begin{split}
\int_{t_1}^{t_2} \left( \int_{B_R} |df_t|^2 |f_t|^2 \varphi^4 \right)^{\frac{2}{3}} ,
 %\int_{t_1}^{t_2} \left( \int_{B_R} |df_t|^3 \varphi^4 \right)^{\frac{4}{5}}, 
 \int_{t_1}^{t_2} \left( \int_{B_R} |f_t|^8 \varphi^4 \right)^{\frac{1}{3}}
% \int_{t_1}^{t_2} \left( \int_{B_R} |df_t|^2 |f_t|^4 \varphi^4 \right)^{\frac{2}{5}} ,
 %\int_{t_1}^{t_2} \left( \int_{B_R} |df_t|^3 \varphi^4 \right)^{\frac{4}{5}}, 
% \int_{t_1}^{t_2} \left( \int_{B_R} |f_t|^{12} \varphi^4 \right)^{\frac{1}{5}}
% \int_{t_1}^{t_2} \left( \int_{B_R} |df_t|^2 |f_t|^6 \varphi^4 \right)^{\frac{2}{7}} ,
 %\int_{t_1}^{t_2} \left( \int_{B_R} |df_t|^3 \varphi^4 \right)^{\frac{4}{5}}, 
% \int_{t_1}^{t_2} \left( \int_{B_R} |f_t|^{16} \varphi^4 \right)^{\frac{1}{7}} 
% \int_{t_1}^{t_2} \left( \int_{B_R} |df_t|^2 |f_t|^4 \varphi^4 \right)^{\frac{2}{5}} ,
%  \int_{t_1}^{t_2} \left( \int_{B_R} |f_t|^{\frac{20}{3}} \varphi^4 \right)^{\frac{3}{7}}
 \leq C_7 (1 + \frac{1}{R^6}) \ep
\end{split}
\end{equation}
for some constant $C_7 = C_7(t_2, t_2-t_1)$.
\end{lemma}

\begin{proof}
First note that by H{\"o}lder and Young's inequality,
\[
\begin{split}
\int_{t_1}^{t_2} \left( \int_{B_R} |df_t|^2 |f_t|^2 \varphi^4 \right)^{\frac{2}{3}} \leq& \int_{t_1}^{t_2} \left( \int_{B_R} |df_t|^4 \varphi^4 \right)^{\frac{1}{3}} \left( \int_{B_R} |f_t|^4 \varphi^4 \right)^{\frac{1}{3}}\\
\leq& \frac{2}{3} \int_{t_1}^{t_2} \left( \int_{B_R} |df_t|^4 \varphi^4 \right)^{\frac{1}{2}} + \frac{1}{3} \int_{t_1}^{t_2} \int_{B_R} |f_t|^4 \varphi^4\\
\leq& C_7 (1 + \frac{1}{R^6}) \ep.
\end{split}
\]

%Next, by integration by parts,
%\[
%\begin{split}
%\int_{B_R} |df_t|^3 \varphi^4 =& -\int_{B_R} \langle \Delta f_t, f_t \rangle |df_t| \varphi^4 - \int_{B_R} \langle df_t, f_t \rangle \langle \nabla df_t, df_t \rangle |df_t|^{-1} \varphi^4 - 4 \int_{B_R} \langle df_t, f_t \rangle |df_t| \varphi^3 \nabla \varphi\\
%\leq& \left(  \int_{B_R} |\nabla df_t|^2 \varphi^4 \right)^{\frac{1}{2}} \left( \int_{B_R} |df_t|^2 |f_t|^2 \varphi^4 \right)^{\frac{1}{2}} + \frac{1}{2} \int_{B_R} |df_t|^3 \varphi^4 + C \int_{B_R} |f_t|^3 \varphi |\nabla \varphi|^3.
%\end{split}
%\]
%Hence, by Young's inequality, we obtain
%\[
%\begin{split}
%\left( \int_{B_R} |df_t|^3 \varphi^4 \right)^{\frac{4}{5}} \leq& C \left(  \int_{B_R} |\nabla df_t|^2 \varphi^4 \right)^{\frac{2}{5}} \left( \int_{B_R} |df_t|^2 |f_t|^2 \varphi^4 \right)^{\frac{2}{5}} + C \left( \int_{B_R} |f_t|^3 \varphi |\nabla \varphi|^3 \right)^{\frac{4}{5}}\\
%\leq& C \int_{B_R} |\nabla df_t|^2 \varphi^4 + C \left( \int_{B_R} |df_t|^2 |f_t|^2 \varphi^4 \right)^{\frac{2}{3}} + C \left( \int_{B_R} |f_t|^3 \varphi |\nabla \varphi|^3 \right)^{\frac{4}{5}}.
%\end{split}
%\]
%Integrating over time, we obtain the desired inequality.
%Similarly,
By Sobolev embedding,
\[
\left( \int_{B_R} |f_t|^8 \varphi^4 \right)^{\frac{1}{3}} \leq C \left( \int_{B_R} |df_t|^2 |f_t|^2 \varphi^2 + |f_t|^4 |\nabla \varphi|^2 \right)^{\frac{2}{3}}.
\]
Integrating over time, we obtain the desired inequality.

\end{proof}

It seems this machinery will boost the exponent further, but due to the $L^p L^q$ structure, higher value of $q$ results lower value of $p$.
We will improve the estimate further in other way, in later section.

%%%%%%%%%%%%%%%%%%%%%%%%%%%%%%%%%%%%%%%%%%%%%%%%%%%%%%%%%%%%%%%%%%%%%%%%%%%%%
\section{$C^0$ estimate for $df$}
\label{sec4}

In this section, we obtain $C^0$ estimate for $|df|$ over $B_r \times [t_1,t_2]$ for some small $r>0$ under smallness assumption of $\int_{B_r} 1 + |\nabla df|^2 + |df|^4$.
Note that by \Cref{loc E cor}, for any $\ep>0$, we can choose $R$ small enough so that $\int_{B_R} |\Delta f|^2 \leq \ep$.
However, this does not guarantee $\int_{B_R} |\nabla df|^2 \leq \ep$ nor $\int_{B_R} |df|^4 \leq \ep$ directly, because $\int |\Delta f|^2$ solely cannot control $\int |\nabla df|^2$ and $\int |df|^4$ locally.
This is one of the reasons why in many literatures of biharmonic map flow, we need $\int_{B_R} |\Delta f|^2 + \left( \int_{B_R} |df|^4 \right)^{\frac{1}{2}} < \ep_0$ or $\int_{B_R} |\nabla df|^2 + \left( \int_{B_R} |df|^4 \right)^{\frac{1}{2}} < \ep_0$ for $\ep$-regularity.

However, we can establish H{\"o}lder-like behavior of them as we will see below.
Note that this behavior is similar to that in n-Conformal Heat Flow, see in \cite{P25}.

\begin{prop} \label{df^4 lem}
Let $(f,u)$ be a smooth solution of \eqref{eq1} on $M \times [t_1,t_2]$ and $\ep>0$ is given.
Then there exists $R = R(\ep)>0$ and $\alpha = \alpha(t_2,R) \in (0,1)$ such that for any $r \in (0,R]$,
\begin{equation} \label{df^4}
\int_{B_r} |df|^4 \leq C \ep^2 + \left( \frac{r}{R} \right)^{\alpha} (4C_1 E(0)^2 + C_2).
\end{equation}
\end{prop}

\begin{proof}
First choose $R_1$ as in \Cref{loc E cor}.
Let $R \leq R_1$ which is determined later and consider $\varphi$ a cut-off function on $B_R$ with $\varphi \equiv 1$ on $B_{R/2}$ and $|\nabla \varphi| \leq \frac{4}{R}$ on $B_R$.
Then by \Cref{loc E cor}, we have $\int_{B_R} |\Delta f|^2 \leq \ep$.

By Sovolev embedding and \eqref{nabla sq}, we have
\[
\begin{split}
\left( \int_{B_R} |df|^4 \varphi^4 \right)^{\frac{1}{2}} \leq& C \left( \int_{B_R} |\nabla df|^2 \varphi^2 + |df|^2 |\nabla \varphi|^2 \right)\\
\leq& C \int_{B_R} |\Delta f|^2 \varphi^2 + C \int_{B_R} |df|^2 \varphi^2 + C \int_{B_R} |df|^2 |\nabla \varphi|^2.
\end{split}
\]
Here the second term is estimated by
\[
C\int_{B_R} |df|^2 \varphi^2 \leq C\left( \int_{B_R} 1 \right)^{\frac{1}{2}} \left( \int_{B_R} |df|^4 \varphi^4 \right)^{\frac{1}{2}} \leq \frac{1}{2} \left( \int_{B_R} |df|^4 \varphi^4 \right)^{\frac{1}{2}}
\]
if $R$ is small enough.
Also, we can estimate the last term by
\[
C \int_{B_R} |df|^2 |\nabla \varphi|^2 \leq C \left( \int_{B_R} |\nabla \varphi|^4 \right)^{\frac{1}{2}} \left( \int_{B_R \setminus B_{R/2}} |df|^4 \right)^{\frac{1}{2}} \leq C \left( \int_{B_R \setminus B_{R/2}} |df|^4 \right)^{\frac{1}{2}}.
\]
Combining them together, we obtain
\[
\int_{B_{R/2}} |df|^4 \leq C \ep^2 + C_0 \int_{B_R \setminus B_{R/2}} |df|^4
\]
for some constant $C_0$ which only depends on the Ricci curvature of $g_0$.
By suitable modification, we obtain
\[
\begin{split}
%\int_{B_{R/2}} |df|^4 \leq& C \ep^2 + C_0 \int_{B_R \setminus B_{R/2}} |df|^4\\
\int_{B_{R/2}} |df|^4 \leq& \frac{C}{1+C_0} \ep^2 + \theta \int_{B_R} |df|^4
\end{split}
\]
where $\theta = \frac{C_0}{1+C_0} \in (0,1)$ is independent on $R$.
Iterating this process and we can obtain
\[
\int_{B_{R/2^k}} |df|^4 \leq \frac{C}{1+C_0} \ep^2 \frac{1}{1-\theta} + \theta^k \int_{B_R} |df|^4.
\]
Hence, for any $r \in (0,R]$, by letting $\alpha = -\log_2(\theta)>0$, we get
\[
\int_{B_r} |df|^4 \leq C \ep^2 + \left( \frac{r}{R} \right)^{\alpha} \int_{B_R} |df|^4 \leq C \ep^2 + \left( \frac{r}{R} \right)^{\alpha} (4C_1 E(0)^2 + C_2).
\]
This completes the proof.
\end{proof}

\begin{cor} \label{nabla df lem}
Let $(f,u)$ be a smooth solution of \eqref{eq1} on $M \times [t_1,t_2]$ and $\ep>0$ is given.
Then there exists $R = R(\ep)>0$ and $\alpha = \alpha(t_2,R) \in (0,1)$ such that for any $r \in (0,R]$,
\begin{equation} \label{nabla df}
\int_{B_r} |\nabla df|^2 \leq C \ep + C \left( \frac{r}{R} \right)^{\frac{\alpha}{2}} (4C_1 E(0)^2 + C_2)^{\frac{1}{2}}.
\end{equation}
\end{cor}

\begin{proof}
Choose $R$ as in \Cref{df^4 lem}.
For any $r \in (0,R]$ and cut-off function $\varphi$ on $B_r$, by \eqref{nabla sq},
\[
\begin{split}
\int_{B_r} |\nabla df|^2 \varphi^2 \leq& 2 \int_{B_r} |\Delta f|^2 \varphi^2 + C \int_{B_r} |df|^2 (\varphi^2 + |\nabla \varphi|^2)\\
\leq& 2 \ep + C \left( \int_{B_r} (\varphi^2 + |\nabla \varphi|^2)^2 \right)^{\frac{1}{2}} \left( \int{B_r} |df|^4 \right)^{\frac{1}{2}}\\
\leq& C \ep + C \left( \frac{r}{R} \right)^{\frac{\alpha}{2}} (4C_1 E(0)^2 + C_2)^{\frac{1}{2}}.
\end{split}
\]
This completes the proof.
\end{proof}

\Cref{df^4 lem} and \Cref{nabla df lem} implies that, for any $\ep>0$, if we choose $r$ small enough, we can guarantee that
\begin{equation} \label{df^4 nabla df^2 ep}
\int_{B_r} |df|^4 \leq C_8 \ep^2, \quad  \int_{B_r} |\nabla df|^2 \leq C_8 \ep
\end{equation}
where the constant $C_8 = C_8(t_2,t_2-t_1,C_1,C_2,E(0),\alpha)$ is independent on $R$.

Using this smallness condition, we can obtain higher order estimate for spatial derivatives.

\begin{prop} \label{nabla^3 f est}
There exists $\ep_1>0$ such that the following holds:

Let $(f,u)$ be a smooth solution of \eqref{eq1} on $M \times [t_1,t_2]$ with
\begin{equation} \label{ass}
\sup_{t \in [t_1,t_2]} \int_{B_r} 1 + |df|^4 + |\nabla df|^2 \leq \ep_1.
\end{equation}
Then
\[
\begin{split}
\int_{B_r} |\nabla df|^4 \varphi^8, \int_{B_r} |df|^8 \varphi^8 \leq& C_9 \ep_1^2\\
\int_{B_r} |\nabla^2 df|^2 \varphi^4, \int_{B_r} |\nabla \Delta f|^2 \varphi^4 \leq& C_9 \ep_1
\end{split}
\]
for some constant
\[
C_9 := C_9(t_2,r,t_2-t_1, V_1(t_1))
\]
where
\[
\begin{split}
%K_1(t_1) =& \int_{B_R} e^{4u}|f_t|^2 (t_1) + 4a \int_{B_R} |\Delta f|^2 (t_!)\\
V_1(t_1) =& \int_{B_r} e^{8u} (t_1).
\end{split}
\]
\end{prop}

\begin{proof}
Combining \eqref{ass} with \eqref{nabla sq f} and \eqref{df^6}, we have
\[
\begin{split}
\int_{B_r} |\nabla df|^2 |df|^2 \varphi^4 \leq& C \left( \int_{B_r} |df|^4 \right)^{\frac{1}{2}} \left( \int_{B_r} | \nabla^2 df|^2 \varphi^4 +  |\nabla df|^2 \varphi^2 |\nabla \varphi|^2 \right)\\
\leq& C \ep_1^{\frac{1}{2}} \int_{B_r} |\nabla^2 df|^2 \varphi^4 + C\frac{1}{r^2} \ep_1^{\frac{3}{2}}\\
\int_{B_r} |df|^6 \varphi^4 \leq & C \left( \int_{B_r} |df|^4 \right) \left( \int_{B_r} | \nabla^2 df|^2 \varphi^4 + |\nabla df|^2 \varphi^2 |\nabla \varphi|^2 \right)\\
&+ C \left( \int_{B_r} |df|^4 \right)^{\frac{1}{2}} \left( \int_{B_r} |df|^4 \varphi^2 |\nabla \varphi|^2 \right)\\
\leq& C \ep_1  \int_{B_r} |\nabla^2 df|^2 \varphi^4 + C \frac{1}{r^2} \ep_1^{\frac{3}{2}}.
\end{split}
\]
Combining these with \eqref{nabla^3 f} and \eqref{comm}, we obtain
\[
\begin{split}
\int_{B_r} |\nabla df|^2 |df|^2 \varphi^4 &+ \int_{B_r} |df|^6 \varphi^4\\
 \leq& C \ep_1^{\frac{1}{2}} \int_{B_r} |\nabla^2 df|^2 \varphi^4 + C\frac{1}{r^2} \ep_1^{\frac{3}{2}}\\
\leq& C \ep_1^{\frac{1}{2}} \int_{B_r} |\nabla \Delta f|^2 \varphi^4 + C \ep_1^{\frac{1}{2}} \int_{B_r} (|\nabla df|^2 |df|^2 + |df|^6) \varphi^4 + C (1 + \frac{1}{r^2}) \ep_1.
\end{split}
\]
Hence, for $\ep_1$ small enough, we can have $C \ep_1^{\frac{1}{2}} \leq \frac{1}{2}$ on RHS, and so
\begin{equation} \label{nabla df with ep}
\int_{B_r} |\nabla df|^2 |df|^2 \varphi^4 + \int_{B_r} |df|^6 \varphi^4 \leq C \ep_1^{\frac{1}{2}} \int_{B_r} |\nabla \Delta f|^2 \varphi^4 + C (1 + \frac{1}{r^2}) \ep_1.
\end{equation}
Now combining this with \Cref{loc W32 lem}, we obtain
\[
\begin{split}
\int_{B_r} |\nabla \Delta f|^2 \varphi^4 \leq& C \int_{B_r} |\nabla df|^2 |df|^2 \varphi^4 + C \int_{B_r} |df|^6 \varphi^4 + C \int_{B_r} |\Delta f|^2 \varphi^2 |\nabla \varphi|^2\\
&+ C \int_{B_r} e^{4u} |\Delta f| |f_t| \varphi^4\\
\leq& C \ep_1^{\frac{1}{2}} \int_{B_r} |\nabla \Delta f|^2 \varphi^4 + C (1 + \frac{1}{r^2}) \ep_1 + C \int_{B_r} e^{4u}|\Delta f| |f_t| \varphi^4.
\end{split}
\]
Again, for $\ep_1$ small enough, we can have $C \ep_1^{\frac{1}{2}} \leq \frac{1}{2}$ on RHS, and so
\begin{equation} \label{nabla^3 f with ep}
\int_{B_r} |\nabla^2 df|^2 \varphi^4, \int_{B_r} |\nabla \Delta f|^2 \varphi^4 \leq C (1 + \frac{1}{r^2}) \ep_1 + C \int_{B_r} e^{4u} |\Delta f| |f_t| \varphi^4.
\end{equation}
where estimate for $\int |\nabla^2 df|^2 \varphi^4$ comes from those for $\int |\nabla \Delta f|^2 \varphi^4$ and from \eqref{nabla^3 f} and \eqref{comm}.
At this moment, by Sobolev embedding, \Cref{Sobolev}, and above inequality, we have
\[
\begin{split}
\left( \int_{B_r} |\nabla df|^4 \varphi^8 \right)^{\frac{1}{2}} \leq& C \left( \int_{B_r} |\nabla^2 df|^2 \varphi^4 + |\nabla df|^2 \varphi^2 |\nabla \varphi|^2 \right) \\
\leq& C \left( 1 + \frac{1}{r^2} \right) \ep_1 + C \int_{B_r} e^{4u} |\Delta f| |f_t| \varphi^4\\
\left( \int_{B_r} |df|^8 \varphi^8 \right)^{\frac{1}{2}} \leq& C \left( 1 + \frac{1}{r^2} \right) \ep_1 + C \ep_1^{\frac{1}{2}} \int_{B_r} e^{4u} |\Delta f| |f_t| \varphi^4.
\end{split}
\]

Now we estimate $\int e^{4u} |\Delta f| |f_t| \varphi^4$.
Note that
\begin{equation} \label{nabla^3 f with f_t}
\begin{split}
\int_{B_r} e^{4u} |\Delta f| |f_t| \varphi^4 \leq& C \left( \int_{B_r} e^{4u} |f_t|^2 \right)^{\frac{1}{2}} \left( \int_{B_r} e^{8u} \varphi^8 \right)^{\frac{1}{4}} \left( \int_{B_r} |\Delta f|^4 \varphi^8 \right)^{\frac{1}{4}}\\
\leq& C \ep_1^{\frac{1}{2}} \left( \int_{B_r} e^{8u} \varphi^8 \right)^{\frac{1}{4}} \left( \int_{B_r} |\Delta f|^4 \varphi^8 \right)^{\frac{1}{4}}.
\end{split}
\end{equation}

By \Cref{e^pu},
\[
\begin{split}
\int_{B_r} e^{8u} \varphi^8 (t) \leq& \int_{B_r} e^{8u} \varphi^8 (t_1) + C \int_{t_1}^{t} \int_{B_r} \left( |\nabla df|^2 + |df|^4 \right)^2 \varphi^8\\
\leq& \int_{B_r} e^{8u} \varphi^8 (t_1) + C \int_{t_1}^{t} \int_{B_r} |\nabla df|^4 \varphi^8 + |df|^8 \varphi^8.
%\leq& \int_{B_R} e^{8u} \varphi^8 (t_1) + C \left( 1 + \frac{1}{R^4} \right) \ep_1^2 (t-t_1) + C \int_{t_1}^{t} \left( \int_{B_R} e^{4u} |\Delta f| |f_t| \varphi^4 \right)^2
\end{split}
\]
Now terms are circulating and we can apply Gronwall's inequality.
Let
\[
X(t) = \int_{B_r} |\nabla df|^4 \varphi^8 + |df|^8 \varphi^8 (t).
\]
Then above inequalities imply
\[
\begin{split}
X(t) \leq&  C \left( 1 + \frac{1}{r^4} \right) \ep_1^2 + C \left( \int_{B_R} e^{4u} |\Delta f| |f_t| \varphi^4 \right)^2\\
\leq& C \left( 1 + \frac{1}{r^4} \right) \ep_1^2 + C  \ep_1 \left( \int_{B_r} e^{8u} \varphi^8 \right)^{\frac{1}{2}} \left( \int_{B_r} |\Delta f|^4 \varphi^8 \right)^{\frac{1}{2}}\\
%\leq& C \left( 1 + \frac{1}{R^4} \right) \ep_1^2 + C  \ep_1  X(t)^{\frac{1}{2}} \left( \int_{B_R} e^{8u}\varphi^8 (t_1) + C \int_{t_1}^{t} X(s) ds \right)^{\frac{1}{2}}\\
\leq& C \left( 1 + \frac{1}{r^4} \right) \ep_1^2 + \frac{1}{2} X(t) + C \ep_1^2 \left(  \int_{B_r} e^{8u}\varphi^8 (t_1) + C \int_{t_1}^{t} X(s) ds \right).
\end{split}
\]
Hence, we obtain Gronwall's type inequality
\[
X(t) \leq D_1 + D_2 \int_{t_1}^{t} X(s)ds
\]
with
\[
D_1 = C \left( 1 + \frac{1}{r^4} \right) \ep_1^2 + C \ep_1^2 \int_{B_r} e^{8u} (t_1), \quad D_2 = C \ep_1^2.
\]
This implies, by Gronwall's inequality,
\[
X(t) \leq D_1 \left( 1 + D_2 (t_2-t_1) e^{D_2 (t_2-t_1)} \right) = C_9 \ep_1^2
\]
for some constant $C_9 = C_9(t_2, r, t_2-t_1, V_1(t_1))$.

Finally, for $\int |\nabla \Delta f|^2 \varphi^4$ and $\int |\nabla^2 df|^2 \varphi^4$, from \eqref{nabla^3 f with ep} and \eqref{nabla^3 f with f_t},
\[
\begin{split}
\int_{B_r} |\nabla^2 df|^2 \varphi^4, \int_{B_r} |\nabla \Delta f|^2 \varphi^4 \leq& C (1 + \frac{1}{r^2}) \ep_1 + C \ep_1^{\frac{1}{2}} \left( V_1(t_1) + C (t_2-t_1) C_9 \right)^{\frac{1}{4}} C_9^{\frac{1}{4}} \ep_1^{\frac{1}{2}}\\
\leq & C_9 \ep_1.
\end{split}
\]
(If needed, re-define $C_9$.)
This completes the proof.
\end{proof}

Now we have a better control on local quantities.
Namely, we can now bounds $\int |\nabla df|^4$ and $\int |\nabla^2 df|^2$.
But by Sobolev embedding, bounds for $\int |\nabla df|^4$ is not enough to guarantee $\sup |df|<C$.
In \Cref{nabla^3 f nabla^2 f est}, we will obtain bounds for $\int |\nabla df|^6$ which is now enough to imply $\sup |df| < C$.

We first introduce the following algebraic style lemma.

\begin{lemma}
For $H$ a $n \times n$ matrix and $v$ a $n$-dimensional vector, we have
\begin{equation} \label{eq alg}
\left| Mv \right|^2 \leq (n-1) |H|^2 |v|^2
\end{equation}
where $M = H - \tr(H) I$ and $|H|^2 = \sum_{j,k} H_{jk}^2$ is the Frobenius norm squared.
\end{lemma}

\begin{proof}
Denote $\lambda_i$ be the eigenvalues of $H$.
Then $\mu_i = \lambda_i - \sum_{i} \lambda_i$ is the eigenvalues of $M$.
Now using the operator norm $\|M\|_{op}$ with $|Mv|^2 \leq \|M\|_{op}^2 |v|^2$,
\[
\begin{split}
 \left| Mv \right|^2 \leq&  \|M\|_{op}^2 |v|^2 \leq  \left( \max_{i} \mu_i^2 \right) |v|^2 \leq \left( (n-1) \sum_i \lambda_i^2 \right) |v|^2 = (n-1) |H|^2 |v|^2
\end{split}
\]
using Cauchy-Swartz inequality.
\end{proof}

Note that above inequality is still valid for higher order tensor, say $H = H^{(i)}_{jk}$, $v = v^{(i)}_k$ for $i=1, \ldots, N$ with $Mv$ is replaced by $\langle M, v \rangle = \sum_i M^{(i)} v^{(i)}$.

The next lemma is a slightly higher estimate than \Cref{loc W32 lem}.

\begin{lemma} \label{loc W32 lem 2}
(Local $W^{3,2}$ estimate with extra factor $|\nabla df|^p$)
Let $(f,u)$ be a smooth solution of \eqref{eq1} on $M \times [t_1,t_2]$.
Then for any $p<\frac{4}{3}$,
\begin{equation} \label{loc W32-2}
\begin{split}
\int_{M} |\nabla^2 df|^2 |\nabla df|^p \varphi^4 \leq& C \int_{M} e^{4u} |f_t| |\nabla df|^{p+1} \varphi^4 + C \int_{M} (1 + |\nabla df|^p |df|^6 + |\nabla df|^{p+3} ) \varphi^4\\
& + C \int_{M} |\nabla df|^{p+2} \varphi^2 |\nabla \varphi|^2.
\end{split}
\end{equation}
where $C$ depends on $p$.
\end{lemma}

\begin{proof}
Proof for \eqref{loc W32-2} is similar, but requires more elaborate treat for the terms.
Integration by parts and \eqref{comm} gives
\begin{equation} \label{nabla^3 f nabla^2 f 1}
\begin{split}
&\int_{M} |\nabla^2 df|^2 |\nabla df|^p \varphi^4\\
 \leq& -\int_{M} \langle \nabla \Delta df, \nabla df \rangle |\nabla df|^p \varphi^4 + C \int_{M} (|df|^2 |\nabla df| + |\nabla df| + |df|^4 + |df| ) |\nabla df|^{p+1} \varphi^4\\
&- p \int_{M} (\langle \nabla^2 df, \nabla df \rangle)^2 |\nabla df|^{p-2} \varphi^4 + 4 \int_{M} |\nabla^2 df| |\nabla df|^{p+1} \varphi^3 |\nabla \varphi|\\
\leq& \int_{M} |\Delta df|^2 |\nabla df|^p \varphi^4 + p \int_{M} \langle \Delta df, \nabla df \rangle \langle \nabla^2 df, \nabla df \rangle |\nabla df|^{p-2} \varphi^4 \\
& + 4 \int_{M} \langle \Delta df, \nabla df \rangle |\nabla df|^p \varphi^3 \nabla \varphi + C \int_{M} (|df|^2 |\nabla df| + |\nabla df| + |df|^4 + |df| ) |\nabla df|^{p+1} \varphi^4\\
&- p \int_{M} (\langle \nabla^2 df, \nabla df \rangle)^2 |\nabla df|^{p-2} \varphi^4 + 4 \int_{M} |\nabla^2 df| |\nabla df|^{p+1} \varphi^3 |\nabla \varphi|\\
\leq& V + VI + VII  +  \delta \int_{M} |\nabla^2 df|^2 |\nabla df|^p \varphi^4 + C \int_{M} (1 + |\nabla df|^p |df|^6 + |\nabla df|^{p+3} ) \varphi^4\\
& + C \int_{M} |\nabla df|^{p+2} \varphi^2 |\nabla \varphi|^2
\end{split}
\end{equation}
where
\[
\begin{split}
V &= \int_{M} |\Delta df|^2 |\nabla df|^p \varphi^4\\
VI &=  p \int_{M} \langle \Delta df, \nabla df \rangle \langle \nabla^2 df, \nabla df \rangle |\nabla df|^{p-2} \varphi^4 \\
VII &= -p \int_{M} (\langle \nabla^2 df, \nabla df \rangle)^2 |\nabla df|^{p-2} \varphi^4.
\end{split}
\]
Now we have
\[
\begin{split}
V + VI  
%=& \int_{M} |\Delta df|^2 |\nabla df| \varphi^4 + \int_{M} \langle \Delta df, \nabla df \rangle \langle \nabla^2 df, \nabla df \rangle |\nabla df|^{-1} \varphi^4\\
=& \int_{M} \langle \Delta \nabla_k f, \nabla_j \left( \nabla_j \nabla_k f |\nabla df|^p \right) \rangle \varphi^4\\
=& \int_{M} \langle \nabla_k \Delta f + Ric_{ki} \nabla_i f, \nabla_j \left( \nabla_j \nabla_k f |\nabla df|^p \right) \rangle \varphi^4\\
\leq& \int_{M} \langle \nabla \Delta f, \Delta df \rangle |\nabla df|^p \varphi^4 + p \int_{M} \langle \nabla_k \Delta f, \nabla_j \nabla_k f \rangle \langle \nabla_j \nabla df, \nabla df \rangle |\nabla df|^{p-2} \varphi^4\\
&+ C \int_{M} |df| |\Delta df| |\nabla df|^p \varphi^4 + C \int_{M} |df| |\nabla df|^p |\nabla^2 df| \varphi^4\\
\leq& \int_{M} |\nabla \Delta f|^2 |\nabla df|^p \varphi^4 + \int_{M} \langle \nabla_k \Delta f, Ric_{ki} \nabla_i f \rangle |\nabla df|^p \varphi^4\\
&+ p \int_{M} \langle \nabla_k \Delta f, \nabla_j \nabla_k f \rangle \langle \nabla_j \nabla df, \nabla df \rangle |\nabla df|^{p-2} \varphi^4\\
&+C \int_{M} |df| |\Delta df| |\nabla df|^p \varphi^4 + C \int_{M} |df| |\nabla df|^p |\nabla^2 df| \varphi^4.
\end{split}
\]
By the integration by parts, the first term above is
\[
\begin{split}
\int_{M} |\nabla \Delta f|^2 |\nabla df|^p \varphi^4 &= -\int_{M} \langle \Delta^2 f, \Delta f \rangle |\nabla df|^p \varphi^4\\
& - p \int_{M} \langle \nabla_k \Delta f, \Delta f \rangle \langle \nabla_k \nabla df, \nabla df \rangle |\nabla df|^{p-2} \varphi^4\\
&- 4 \int_{M} \langle \nabla_k \Delta f, \Delta f \rangle |\nabla df|^p \varphi^3 \nabla_k \varphi.
\end{split}
\]
Combining these together, we have
\[
\begin{split}
V + VI \leq& -\int_{M} \langle \Delta^2 f, \Delta f \rangle |\nabla df|^p \varphi^4 -  p \int_{M} \langle \nabla_k \Delta f, \Delta f \rangle \langle \nabla_k \nabla df, \nabla df \rangle |\nabla df|^{p-2} \varphi^4\\
&+p \int_{M} \langle \nabla_k \Delta f, \nabla_j \nabla_k f \rangle \langle \nabla_j \nabla df, \nabla df \rangle |\nabla df|^{p-2} \varphi^4\\
& + \delta \int_{M} |\nabla^2 df|^2 |\nabla df|^p \varphi^4 + C \int_{M} (1 + |\nabla df|^p |df|^6 + |\nabla df|^{p+3}) \varphi^4\\
& + C \int_{M} |\nabla df|^{p+2} \varphi^2 |\nabla \varphi|^2\\
=& I_1 + I_2 + I_3 + I_4 + I_5.
\end{split}
\]
Now $I_2, I_3$ above can be estimate by, using $H_{jk} = \nabla_j \nabla_k f$, $v_k = \nabla_k \Delta f$ and $w_k = \langle \nabla_k \nabla df, \nabla df \rangle$,
\[
\begin{split}
I_2 + I_3 =& -p \int_{M} \langle v_k, \tr(H) I \rangle w_k |\nabla df|^{p-2} \varphi^4 + p \int_{M} \langle v_j, H_{jk} \rangle w_j |\nabla df|^{p-2} \varphi^4\\
\leq& \frac{p}{4} \int_{M} \left( \langle v_k, H_{jk} - \tr(H) I \rangle \right)^2 |\nabla df|^{p-2} \varphi^4 + p \int_{M} |w_k|^2 |\nabla df|^{p-2} \varphi^4\\
\leq& \frac{3p}{4} \int_{M} |\nabla \Delta f|^2 |\nabla df|^p \varphi^4 + p \int_{M} |w_k|^2 |\nabla df|^{p-2} \varphi^4
\end{split}
\]
where the last equality is obtained by \eqref{eq alg} with $n=4$.
In conclusion, \eqref{nabla^3 f nabla^2 f 1} becomes
\[
\begin{split}
\int_{M} |\nabla^2 df|^2 |\nabla df|^p \varphi^4 \leq& -\int_{M} \langle \Delta^2 f, \Delta f \rangle |\nabla df|^p \varphi^4 + \left( \frac{3p}{4} + \delta \right) \int_{M} |\nabla^2 df|^2 |\nabla df|^p \varphi^4\\
& + C \int_{M} (1 + |\nabla df|^p |df|^6 + |\nabla df|^{p+3} ) \varphi^4 + C \int_{M} |\nabla df|^{p+2} \varphi^2 |\nabla \varphi|^2.
\end{split}
\]
Since we assume $p<\frac{4}{3}$, the term $\int |\nabla^2 df|^2 |\nabla df|^p \varphi^4$ can be absorbed into the left hand sides for $\delta$ small enough.

On the other hand, multiply the first equation in \eqref{eq1} with $-\Delta f |\nabla df|^p \varphi^4$ to get
\[
\begin{split}
&-\int_{M} \langle \Delta^2 f, \Delta f \rangle |\nabla df|^p \varphi^4\\
  =& \int_{M} e^{4u} \langle f_t, \Delta f \rangle |\nabla df|^p \varphi^4 - \int_{M} \langle \Delta (A(df,df)), \Delta f \rangle |\nabla df|^p \varphi^4\\
&+ \int_{M} \langle \langle \Delta f, \Delta P \rangle, \Delta f \rangle |\nabla df|^p \varphi^4 - 2 \int_{M} \langle \nabla \langle \Delta f, \nabla P \rangle, \Delta f \rangle |\nabla df|^p \varphi^4\\
\leq& \int_{M} e^{4u} |f_t|  |\Delta f| |\nabla df|^p \varphi^4\\
&+  \int_{M} \langle 2A (\nabla df, df) + DA(df,df) \cdot df, \nabla \Delta f |\nabla df|^p \varphi^4 + 4 \Delta f |\nabla df|^p \varphi^3 \nabla \varphi \rangle\\
&+ \int_{M} \langle 2A(\nabla df, df ) + DA(df,df) \cdot df, p \Delta f \langle \nabla^2 df, \nabla df \rangle |\nabla df|^{p-2} \varphi^4 \rangle\\
&+ 3C_N \int_{M} (|\Delta f|^3  + |\Delta f|^2 |df|^2) |\nabla df|^p \varphi^4 + 2C_N \int_{M} |\nabla \Delta f| |\Delta f| |df| |\nabla df|^p \varphi^4\\
=& \int_{M} e^{4u} |f_t|  |\Delta f| |\nabla df|^p \varphi^4 + II_{2} + II'_{2} + III_{2} + 2IV_{2}.
\end{split}
\]

As above, each term can be estimate by
\[
\begin{split}
II_{2} \leq&  C_N \int_{M} (2|\nabla df| |df| + |df|^3) (|\nabla \Delta f| |\nabla df|^p \varphi^4 + 4|\Delta f|  |\nabla df|^p \varphi^3 |\nabla \varphi|) \\
\leq& \delta \int_{M} |\nabla \Delta f|^2 |\nabla df|^p \varphi^4 + C \int_{M} |df|^6 |\nabla df|^p \varphi^4 + C \int_{M} |\nabla df|^{p+2} |df|^2 \varphi^4  \\
&+ C \int_{M} |\Delta f|^2 |\nabla df|^p \varphi^2 |\nabla \varphi|^2\\
II'_{2} \leq& C_N \int_{M} (2 |\nabla df| |df| + |df|^3) (|\nabla^2 df| |\nabla df^p| \varphi^4)\\
\leq& \delta \int_{M} |\nabla^2 df|^2 |\nabla df|^p \varphi^4 + C \int_{M} |\nabla df|^{p+2} |df|^2 \varphi^4 + C \int_{M} |df|^6 |\nabla df|^p \varphi^4\\
III_{2} \leq& C \int_{M} |\nabla df|^{p+3} \varphi^4 + C \int_{M} |\nabla df|^{p+2} |df|^2 \varphi^4\\
IV_{2} \leq& \delta \int_{M} |\nabla^2 df|^2 |\nabla df|^p \varphi^4 + C \int_{M} |\nabla df|^{p+2} |df|^2 \varphi^4.
\end{split}
\]
Combining all together, we obtain
\[
\begin{split}
\int_{M} |\nabla^2 df|^2 |\nabla df|^p \varphi^4 \leq& C \int_{M} e^{4u} |f_t| |\nabla df|^{p+1} \varphi^4 + C \int_{M} (1 + |\nabla df|^p |df|^6 + |\nabla df|^{p+3} ) \varphi^4\\
& + C \int_{M} |\nabla df|^{p+2} \varphi^2 |\nabla \varphi|^2.
\end{split}
\]
This completes the proof.
\end{proof}

\begin{prop} \label{nabla^3 f nabla^2 f est}
There exists $\ep_1>0$ such that the following holds:

Let $(f,u)$ be a smooth solution of \eqref{eq1} on $M \times [t_1,t_2]$ with
\begin{equation*}
\sup_{t \in [t_1,t_2]} \int_{B_r} 1 + |df|^4 + |\nabla df|^2 \leq \ep_1.
\end{equation*}
Then
\[
\begin{split}
\int_{B_r} |\nabla df|^6 \varphi^8, \int_{B_r} |df|^{12} \varphi^8 \leq& C_{10} \ep_1^2 \\
\int_{B_r} |\nabla^2 df|^2 |\nabla df| \varphi^4 \leq& C_{10} \ep_1
\end{split}
\]
for some constant
\[
C_{10} := C_{10}(t_2,r,t_2-t_1, V_2(t_1))
\]
where
\[
V_2(t_1) = \int_{B_r} e^{12u}(t_1).
\]
\end{prop}

\begin{proof}
As above, by \eqref{loc W32-2} with $p=1$ and \Cref{nabla^3 f est},
\begin{equation} \label{nabla^3 f nabla^2 f with ep}
\begin{split}
\int_{B_r} |\nabla^2 df|^2 |\nabla df| \varphi^4 \leq& C \int_{B_r} (1 + |\nabla df|^4 + |df|^8) \varphi^4 + C \int_{B_r} |\nabla df|^3 \varphi^2 |\nabla \varphi|^2\\
& +  C \int_{B_r} e^{4u} |f_t| |\nabla df|^2 \varphi^4\\
\leq& C \ep_1 + C \frac{1}{r^2} \ep_1^{\frac{3}{2}} + C \int_{B_r} e^{4u} |f_t| |\nabla df|^2 \varphi^4.
\end{split}
\end{equation}
Here we use interpolation inequality
\[
\int_{B_r} |\nabla df|^3 \varphi^2 |\nabla \varphi|^2 \leq C \frac{1}{r^2} \left( \int_{B_r} |\nabla df|^2 \right)^{\frac{1}{2}} \left( \int_{B_r} |\nabla df|^4 \varphi^4 \right)^{\frac{1}{2}} \leq C \frac{1}{r^2} \ep_1^{\frac{3}{2}}.
\]
For the last term, we can estimate it by
\begin{equation} \label{nabla^3 f nabla^2 f with f_t}
\begin{split}
\int_{B_r} e^{4u} |f_t| |\nabla df|^2 \varphi^4 \leq& C \left( \int_{B_r} e^{4u} |f_t|^2 \right)^{\frac{1}{2}} \left( \int_{B_r} |\nabla df|^6 \varphi^8 \right)^{\frac{1}{3}} \left( \int_{B_r} e^{12u} \varphi^8 \right)^{\frac{1}{6}}\\
\leq& C \ep_1^{\frac{1}{2}} \left( \int_{B_r} |\nabla df|^6 \varphi^8 \right)^{\frac{1}{3}} \left( \int_{B_r} e^{12u} \varphi^8 \right)^{\frac{1}{6}}.
\end{split}
\end{equation}
By \Cref{e^pu},
\[
\begin{split}
\int_{B_r} e^{12u} \varphi^8 (t) \leq& \int_{B_r} e^{12u} \varphi^8 (t_1) + C \int_{t_1}^{t} \int_{B_r} \left( |\nabla df|^2 + |df|^4 \right)^3 \varphi^8\\
\leq& \int_{B_r} e^{12u} \varphi^8 (t_1) + C \int_{t_1}^{t} \int_{B_r} |\nabla df|^6 \varphi^8 + |df|^{12} \varphi^8.
\end{split}
\]
On the other hand, by Sobolev embedding and above inequalities, we have
\[
\begin{split}
\left( \int_{B_r} |\nabla df|^6 \varphi^8 \right)^{\frac{1}{2}} \leq& C \left( \int_{B_r} |\nabla^2 df|^2 |\nabla df| \varphi^4 + |\nabla df|^3 \varphi^2 |\nabla \varphi|^2 \right)\\
\leq& C \ep_1^{\frac{3}{2}} + C \int_{B_r} |\nabla^2 df|^2 |\nabla df| \varphi^4,
\end{split}
\]
\[
\begin{split}
\left( \int_{B_r} |df|^{12} \varphi^8 \right)^{\frac{1}{2}} \leq& C \left( \int_{B_r} |\nabla df|^2 |df|^4 \varphi^4 + |df|^6 \varphi^2 |\nabla \varphi|^2 \right)\\
\leq&C \left( \int_{B_r} |\nabla df|^6 \varphi^8 \right)^{\frac{1}{3}} \left( \int_{B_r} |df|^6 \varphi^2 \right)^{\frac{2}{3}} + C \int_{B_r} |df|^6 \varphi^2 |\nabla \varphi|^2\\
\leq& C \left( \int_{B_r} |\nabla df|^6 \varphi^8 \right)^{\frac{1}{2}} + C \left( \int_{B_r} |df|^6 \varphi^2 \right)^2 + C \int_{B_r} |df|^6 \varphi^2 |\nabla \varphi|^2\\
\leq& C \ep_1^{\frac{3}{2}} + C \int_{B_r} |\nabla^2 df|^2 |\nabla df| \varphi^4 + C \ep_1^3 +C \frac{1}{r^2} \ep_1^{\frac{3}{2}}
\end{split}
\]
where we use the following interpolation inequality
\[
\int_{B_r} |df|^6 \leq \left( \int_{B_r} |df|^4 \right)^{\frac{1}{2}} \left( \int_{B_r} |df|^8 \right)^{\frac{1}{2}} \leq \ep_1^{\frac{3}{2}}.
\]

Combining above inequalities, we obtain
\[
\begin{split}
\int_{B_r} \left( |\nabla df|^6 + |df|^{12} \right) \varphi^8 \leq& C (1 + \frac{1}{r^4}) \ep_1^3 + C \left( \int_{B_r} |\nabla^2 df|^2 |\nabla df| \varphi^4 \right)^2\\
\leq& C \left(1 + \frac{1}{r^4} \right) \ep_1^2 + C \ep_1 \left( \int_{B_r} |\nabla df|^6 \varphi^8 \right)^{\frac{2}{3}} \left( \int_{B_r} e^{12u} \varphi^8 \right)^{\frac{1}{3}}\\
\leq& C \left(1 + \frac{1}{r^4} \right) \ep_1^2 + \frac{1}{2}  \int_{B_r} |\nabla df|^6 \varphi^8  \\
&+ C \ep_1^3 \left( \int_{B_r} e^{12u}\varphi^8 (t_1) + C \int_{t_1}^{t} \int_{B_r} \left( |\nabla df|^6 + |df|^{12} \right) \varphi^8 \right).
\end{split}
\]
Now we denote
\[
X(t) = \int_{B_r} \left( |\nabla df|^6 + |df|^{12} \right) \varphi^8 (t).
\]
Then the above inequality becomes
\[
X(t) \leq D_1 + D_2 \int_{t_1}^{t} X(s)ds
\]
where
\[
D_1 = C \left( 1 + \frac{1}{r^4} \right) \ep_1^2 + C \ep_1^3 \int_{B_r} e^{12u} (t_1), \quad D_2 = C \ep_1^3.
\]
Then by Gronwall's inequality, we obtain
\[
X(t) \leq D_1 \left(1 + D_2 (t_2-t_1) e^{D_2(t_2-t_1)} \right) = C_{10} \ep_1^2
\]
for some constant $C_{10} = C_{10}(t_2,r,t_2-t_1, V_2(t_1))$.

Finally, for $\int |\nabla^2 df|^2 |\nabla df| \varphi^4$, from \eqref{nabla^3 f nabla^2 f with ep} and \eqref{nabla^3 f nabla^2 f with f_t},
\[
\begin{split}
\int_{B_r} |\nabla^2 df|^2 |\nabla df| \varphi^4 \leq& C \left( 1 + \frac{1}{r^2} \right) \ep_1 + C \ep_1^{\frac{1}{2}} (V_2(t_1) + C (t_2-t_1) C_{10} )^{\frac{1}{6}} C_{10}^{\frac{1}{3}} \ep_1\\
\leq& C_{10} \ep_1.
\end{split}
\]
This completes the proof.
\end{proof}

As a consequence, we obtain
\begin{equation} \label{df_0 est}
\int_{B_r} e^{12u} \varphi^8(t) , \sup_{t \in [t_1,t_2]} \|df\|_{C^0(B_r)} \leq C_{11}
\end{equation}
for some constant $C_{11} = C_{11}(t_2,r,t_2-t_1,V_2(t_1), \ep_1)$.

%%%%%%%%%%%%%%%%%%%%%%%%%%%%%%%%%%%%%%%%%%%%%%%%%%%%%%%%%%%%%%%%%%%%%%%%%%%%%
\section{Higher order estimate}
\label{sec5}

In this section, we obtain H{\"o}lder estimate for the solution under the smallness assumption \eqref{ass}.
%Note that as a result of previous section, we have $\sup |df| \leq C_{10}$.
Throughout this section, $\varphi$ is a cut-off function on $B_r$ with $|\nabla \varphi| \leq \frac{4}{r}$.

We first show the following $L^2 W^{4,2}$-type estimate.
\begin{prop} \label{nabla^4 f est}
There exists $\ep_1>0$ such that the following holds:

Let $(f,u)$ be a smooth solution of \eqref{eq1} on $M \times [t_1,t_2]$ with
\begin{equation*}
\sup_{t \in [t_1,t_2]} \int_{B_r} 1 + |df|^4 + |\nabla df|^2 \leq \ep_1.
\end{equation*}
Then
\begin{equation} \label{loc W42}
\int_{t_1}^{t_2} \int_{B_r} |\Delta^2 f|^2 \varphi^4, \int_{t_1}^{t_2} \left( \int_{B_r} |\nabla \Delta f|^4 \varphi^4 \right)^{\frac{1}{2}} \leq C_{12}
\end{equation}
for some constant
\[
C_{12} := C_{12}(t_2,r,t_2-t_1, V_2(t_1),\ep_1).
\]
\end{prop}

\begin{proof}
From the equation $e^{4u} f_t = -\Delta^2 f + B$ and $f_t \perp B$, we have
\[
\int_{B_r} |\Delta^2 f|^2 \varphi^4 = \int_{B_r} e^{8u} |f_t|^2 \varphi^4 + \int_{B_r} |B|^2 \varphi^4.
\]
And for the second term, we note that
\[
\begin{split}
B =& \Delta (A(df,df)) - \langle \Delta f, \Delta P \rangle + 2 \nabla \langle \Delta f, \nabla P \rangle\\
|B|^2 \leq& C ( |\Delta df|^2 |df|^2 + |\nabla df|^4 + |\nabla df|^2 |df|^4 + |df|^8 \\
& \quad + |\Delta f|^4 + |\Delta f|^2 |df|^4 + |\nabla \Delta f|^2 |df|^2 + |\nabla df|^2 |\Delta f|^2 ).
\end{split}
\]
As a consequence of previous section and \eqref{df_0 est}, we obtain that $\int_{B_r} |B|^2 \varphi^4 \leq C$ for some constant $C$.
Next, for the first term, we obtain
\[
\begin{split}
\int_{B_r} e^{8u} |f_t|^2 \varphi^4 = & \left( \int_{B_r} e^{12u} \varphi^4 \right)^{\frac{2}{3}} \left( \int_{B_r} |f_t|^6 \varphi^4 \right)^{\frac{1}{3}} \leq \frac{2}{3} \left( \int_{B_r} |f_t|^6 \varphi^4 \right)^{\frac{1}{2}} + \frac{1}{3} \left( \int_{B_r} e^{12u} \varphi^4 \right)^2.
\end{split}
\]
By \Cref{f_t^4 with ep 2} and \eqref{df_0 est}, we obtain the first estimate.

The second estimate comes from Sobolev embedding $W^{1,2}_0 \hookrightarrow L^4$.
\end{proof}

Combining above result with \Cref{nabla^3 f est}, we get
\[
\int_{t_1}^{t_2} \int_{B_r} |\nabla \Delta f|^3 \varphi^4 \leq \int_{t_1}^{t_2} \left( \int_{B_r} |\nabla \Delta f|^4 \varphi^4 \right)^{\frac{1}{2}} \left( \int_{B_r} |\nabla \Delta f|^2 \varphi^4 \right)^{\frac{1}{2}} \leq C_{12} C_9^{\frac{1}{2}} \ep_1^{\frac{1}{2}}.
\]
Instead, combining with \Cref{nabla^3 f nabla^2 f est}, we also get
\[
\int_{t_1}^{t_2} \int_{B_r} |\nabla \Delta f|^2 |\nabla df|^3 \varphi^4 \leq \int_{t_1}^{t_2} \left( \int_{B_r} |\nabla \Delta f|^4 \varphi^4 \right)^{\frac{1}{2}} \left( \int_{B_r} |\nabla df|^6 \varphi^4 \right)^{\frac{1}{2}} \leq C_{12} C_{10}^{\frac{1}{2}} \ep_1^{\frac{1}{2}}. 
\]

At this moment, we need more elaborate estimate than \Cref{Int by parts general} to relate $\iint |\nabla \Delta f|^2 |\nabla df|^p \varphi^4$ and $\iint |\nabla^2 df|^2 |\nabla df|^p \varphi^4$ for any $p>0$, under the bounds for $\iint |\nabla df|^{p+2} \varphi^4$.

\begin{lemma} \label{nabla^3 f nabla^2 f p comm est}
There exists $\ep_1>0$ such that the following holds:

Let $(f,u)$ be a smooth solution of \eqref{eq1} on $M \times [t_1,t_2]$ with
\begin{equation*}
\sup_{t \in [t_1,t_2]} \int_{B_r} 1 + |df|^4 + |\nabla df|^2 \leq \ep_1.
\end{equation*}
Also assume that for some $p>0$,
\[
\int_{t_1}^{t_2} \int_{B_r} |\nabla df|^{p+2} \leq C'.
\]
Then
\begin{equation}
\int_{t_1}^{t_2} \int_{B_r} |\nabla^2 df|^2 |\nabla df|^p \varphi^4 \leq (2 + p) \int_{t_1}^{t_2} \int_{B_r} |\nabla \Delta f|^2 |\nabla df|^p \varphi^4 + C_{p}'
\end{equation}
for some constant $C_{p}' := C_{p}'(t_2,r,t_2-t_1,V_2(t_1),\ep_1,p,C')$.
\end{lemma}

\begin{proof}
By integration by parts and \eqref{comm},
\[
\begin{split}
\int_{B_r} |\nabla^2 df|^2 |\nabla df|^p \varphi^4 =& -\int_{B_r} \langle \Delta \nabla df, \nabla df \rangle |\nabla df|^p \varphi^4 - p \int_{B_r} (\langle \nabla^2 df, \nabla df \rangle)^2 |\nabla df|^{p-2} \varphi^4\\
& - 4 \int_{B_r} \langle \nabla^2 df, \nabla df \rangle |\nabla df|^p \varphi^3 \nabla \varphi\\
\leq& -\int_{B_r} \langle \nabla \Delta df, \nabla df \rangle |\nabla df|^p \varphi^4 - p \int_{B_r} (\langle \nabla^2 df, \nabla df \rangle)^2 |\nabla df|^{p-2} \varphi^4 \\
&+ C \int_{B_r} (|\nabla df| + C ) |\nabla df|^p \varphi^4 + C \int_{B_r} |\nabla^2 df| |\nabla df|^{p+1} \varphi^3 |\nabla \varphi|\\
=& \int_{B_r} |\Delta df|^2 |\nabla df|^p \varphi^4 +p \int_{B_r} \langle \Delta df, \nabla df \rangle \langle \nabla^2 df, \nabla df \rangle |\nabla df|^{p-2} \varphi^4\\
&- p \int_{B_r} (\langle \nabla^2 df, \nabla df \rangle)^2 |\nabla df|^{p-2} \varphi^4 \\
&+ C \int_{B_r} (|\nabla df| + C ) |\nabla df|^p \varphi^4 + C \int_{B_r} |\nabla^2 df| |\nabla df|^{p+1} \varphi^3 |\nabla \varphi|\\
\leq& \left( 1 + \frac{p}{4} \right) \int_{B_r} |\Delta df|^2 |\nabla df|^p \varphi^4 + \delta \int_{B_r} |\nabla^2 df|^2 |\nabla df|^p \varphi^4\\
&+ C \int_{B_r} |\nabla df|^{p+2} \varphi^2 |\nabla \varphi|^2.
\end{split}
\]
Using \eqref{comm} again and integrating over $[t_1,t_2]$, we obtain the desired inequality.
\end{proof}

By \Cref{nabla^3 f nabla^2 f p comm est} with $p=3$, we have
\[
\int_{t_1}^{t_2} \int_{B_r} |\nabla^2 df|^2 |\nabla df|^3 \varphi^4 \leq 5C_{12} C_{10}^{\frac{1}{2}} \ep_1^{\frac{1}{2}} + C_{3}'.
\]

By Sobolev embedding and H{\"o}lder inequality, we also have
\begin{align}
\left( \int_{B_r} |\nabla df|^{10} \varphi^8 \right)^{\frac{1}{2}} \leq& C \int_{B_r} |\nabla^2 df|^2 |\nabla df|^3 \varphi^4 + |\nabla df|^5 \varphi^2 |\nabla \varphi|^2 \nonumber\\
%\int_{t_1}^{t_2} \left( \int_{B_r} |\Delta f|^{10} \varphi^8 \right)^{\frac{1}{2}} \leq& C_{13}\\
\int_{t_1}^{t_2} \int_{B_r} |\nabla df|^8 \varphi^4 \leq& \int_{t_1}^{t_2} \left( \int_{B_r} |\nabla df|^{10} \varphi^8 \right)^{\frac{1}{2}} \left( \int_{B_r} |\nabla df|^6  \right)^{\frac{1}{2}} \leq C_{13} \label{W28}
\end{align}
which also implies
\begin{equation} \label{e^16}
\int_{B_r} e^{16 u}\varphi^8 (t) \leq \int_{B_r} e^{16u} \varphi^8 (t_1) + C \int_{t_1}^{t} \int_{B_r} |\nabla df|^8 \varphi^8 + |df|^{16} \varphi^8 \leq C_{13}
\end{equation}
for some constant $C_{13} := C_{13} (t_2,r,t_2-t_1, V_3(t_1),\ep_1)$ where
\[
V_3(t_1) := \int_{B_r} e^{16 u} (t_1).
\]

With the help of improved control for $\int e^{16 u}$, we get the following 
%$L^{\frac{8}{3}} W^{4,\frac{8}{3}}$ estimate.
higher order estimate.
\begin{prop} \label{nabla^4 f est 2}
There exists $\ep_1>0$ such that the following holds:

Let $(f,u)$ be a smooth solution of \eqref{eq1} on $M \times [t_1,t_2]$ with
\begin{equation*}
\sup_{t \in [t_1,t_2]} \int_{B_r} 1 + |df|^4 + |\nabla df|^2 \leq \ep_1.
\end{equation*}
Then
\begin{equation} \label{loc W4 8/3}
\int_{t_1}^{t_2} \int_{B_r} |\Delta^2 f|^{\frac{8}{3}} \varphi^4
%, \int_{t_1}^{t_2} \left( \int_{B_r} |\nabla \Delta f|^{8} \varphi^4 \right)^{\frac{1}{3}} 
\leq C_{14}
\end{equation}
for some constant
\[
C_{14} := C_{14}(t_2,r,t_2-t_1, V_3(t_1),\ep_1).
\]
\end{prop}

\begin{proof}
As above, from $|\Delta^2 f|^2 = e^{8u} |f_t|^2 + |B|^2$, we have
\[
\int_{B_r} |\Delta^2 f|^{\frac{8}{3}} \varphi^4 \leq C \int_{B_r} e^{\frac{32}{3}u} |f_t|^{\frac{8}{3}} \varphi^4 + C \int_{B_r} |B|^{\frac{8}{3}} \varphi^4.
\]
For the second term, using \eqref{comm}, we have
\[
|B|^{\frac{8}{3}} \leq C \left( |\nabla \Delta f|^{\frac{8}{3}} |df|^{\frac{8}{3}} + |\nabla df|^{\frac{16}{3}} + |\nabla df|^{\frac{8}{3}} |df|^{\frac{16}{3}} + |df|^{\frac{32}{3}} \right).
\]
Hence, as a consequence of previous section and above computations, $\int_{t_1}^{t_2} \int_{B_r} |B|^{\frac{8}{3}} \varphi^4 \leq C$.
For the first term, we have
\[
\int_{B_r} e^{\frac{32}{3}u} |f_t|^{\frac{8}{3}} \varphi^4 \leq \left( \int_{B_r} e^{16u} \varphi^4 \right)^{\frac{2}{3}} \left( \int_{B_r} |f_t|^{8} \varphi^4 \right)^{\frac{1}{3}}.
\]
By \eqref{e^16} and \Cref{f_t^4 with ep 3}, we get the desired inequality.
%
%The second estimate comes from Sobolev embedding $W^{1,\frac{8}{3}}_0 \hookrightarrow L^8$.
\end{proof}

%Now, from the elliptic regularity $\|f\|_{W^{4,p}} \leq C \left( \|\Delta^2 f\|_{L^p} + \|f\|_{L^p} \right)$, for $p=\frac{8}{3}$, we get
By the standard elliptic regularity, we get
\[
\int_{t_1}^{t_2} \int_{B_r} |\nabla^3 df|^{\frac{8}{3}} \varphi^4 \leq C_{15}
\]
for some constant $C_{15} := C_{15}(t_2,r,t_2-t_1,V_3(t_1),\ep_1)$.
This implies that, by Sobolev embedding,
\begin{equation} \label{W38}
\int_{t_1}^{t_2} \left( \int_{B_r} |\nabla^2 df|^{8} \varphi^4 \right)^{\frac{1}{3}}, \int_{t_1}^{t_2} (\sup_{B_r} |\nabla df|)^{\frac{8}{3}} dt \leq C_{15}.
\end{equation}
And it is now enough to get the following uniform estimate for $u$.

\begin{cor} \label{sup u}
There exists $\ep_1>0$ such that the following holds:

Let $(f,u)$ be a smooth solution of \eqref{eq1} on $M \times [t_1,t_2]$ with
\begin{equation*}
\sup_{t \in [t_1,t_2]} \int_{B_r} 1 + |df|^4 + |\nabla df|^2 \leq \ep_1.
\end{equation*}
Then
\begin{equation}
\sup_{B_r \times [t_1,t_2]}u(x,t) \leq C_{u}
\end{equation}
for some constant $C_{u} := C_{u}(t_2,r,t_2-t_1,V_u(t_1),\ep_1)$
where
\[
V_u(t_1) := \sup_{B_r} u (x,t_1).
\]
\end{cor}

\begin{proof}
First note that we have a variant of \eqref{e 4u} for $t \in [t_1,t_2]$ by
\[
e^{4u}(t) = e^{-4a(t-t_1)} \left( e^{4u}(t_1) + 4b \int_{t_1}^{t} e^{4a(s-t_1)} (|\nabla df|^2(s) + |df|^4(s)) ds \right).
\]
Denote $V_u(t_1)$ by
\[
V_u(t_1) := \sup_{B_r} u (x,t_1).
\]
Then we have
\[
\begin{split}
e^{4u}(t) &\leq e^{4V_u(t_1)} + 4b e^{4a (t_2-t_1)} \int_{0}^{t} (\sup_{B_r} |\nabla df| )^2 (s) ds + C
\end{split}
\]
which is bounded.
\end{proof}

The above result is crucial in the sense that the operator 
%$e^{4u} \partial_t + \Delta^2$, or equivalently 
$\partial_t + e^{-4u} \Delta^2$ is now uniformly parabolic.

%Moreover, we obtain the following $W^{4,2}$ estimate in free as well.
%
%\begin{prop} \label{W42 prop}
%There exists $\ep_1>0$ such that the following holds:
%
%Let $(f,u)$ be a smooth solution of \eqref{eq1} on $M \times [t_1,t_2]$ with
%\begin{equation*}
%\sup_{t \in [t_1,t_2]} \int_{B_r} 1 + |df|^4 + |\nabla df|^2 \leq \ep_1.
%\end{equation*}
%Then
%\begin{equation}
%\int_{B_r} |\nabla^3 df|^2 \varphi^4 \leq C_{16}
%\end{equation}
%for some constant $C_{16} := C_{16}(t_2,r,t_2-t_1,V_u(t_1),\ep_1)$.
%\end{prop}
%
%\begin{proof}
%As in the proof of \Cref{nabla^4 f est}, we know that $\int_{B_r} |B|^2 \varphi^4$ is bounded.
%Also,
%\[
%\int_{B_r} e^{8u} |f_t|^2 \varphi^4 \leq e^{4 C_{u}} \int_{B_r} e^{4u} |f_t|^2 \varphi^4 \leq C.
%\]
%Hence we have
%\[
%\int_{B_r} |\Delta^2 f|^2 \varphi^4 \leq C.
%\]
%Finally, by standard elliptic regularity, we complete the proof.
%\end{proof}

Now we are ready to prove smoothness.

\begin{theorem} \label{loc smooth}
There exists $\ep_1>0$ such that the following holds:

Let $(f,u)$ be a smooth solution of \eqref{eq1} on $M \times [t_1,t_2]$ with
\begin{equation*}
\sup_{t \in [t_1,t_2]} \int_{B_r} 1 + |df|^4 + |\nabla df|^2 \leq \ep_1.
\end{equation*}
Then H{\"o}lder norms of $f,u$ and their derivatives are all bounded by constants only depending on $t_2,r,t_2-t_1,\ep_1$ and corresponding norms of $u$ at $t_1$.
\end{theorem}

\begin{proof}
Note that from \Cref{nabla^3 f est} and \eqref{W38}, we have
\[
\int_{t_1}^{t_2} \int_{B_r} |\nabla^2 df|^4 \varphi^4 \leq \int_{t_1}^{t_2} \left( \int_{B_r} |\nabla^2 df|^8 \varphi^4 \right)^{\frac{1}{3}} \left( \int_{B_r} |\nabla^2 df|^2 \varphi^4 \right)^{\frac{2}{3}} \leq C_9^{\frac{2}{3}} \ep_1^{\frac{2}{3}} C_{15}
\]
hence $|\nabla^2 df| \in L^4(B_r \times [t_1,t_2])$.

Now we focus on uniformly parabolic operator $\partial_t + e^{-4u} \Delta^2$.
We have
\[
|(\partial_t + e^{-4u} \Delta^2) f| \leq e^{-4u} |B| \in L^4(B_r \times [t_1,t_2])
\]
which implies $\nabla^3 df, \partial_t f \in L^4(B_r \times [t_1,t_2])$ by standard parabolic theory.
Then by Sobolev embedding, we have
\[
\int_{t_1}^{t_2} \left( \int_{B_r} |\nabla^2 df|^{p} \varphi^4 \right)^{\frac{4}{p}} \leq C \int_{t_1}^{t_2} \int_{B_r} |\nabla^3 df|^4 \varphi^4 \leq C
\]
for any $p>1$.
As above, we get
\[
\int_{t_1}^{t_2} \int_{B_r} |\nabla^2 df|^{6-\frac{8}{p}} \varphi^4 \leq \int_{t_1}^{t_2} \left( \int_{B_r} |\nabla^2 df|^p \varphi^4 \right)^{\frac{4}{p}} \left( \int_{B_r} |\nabla^2 df|^2 \varphi^4 \right)^{\frac{p-4}{p}} \leq C.
\]
This implies that $\nabla^3 df, \partial_t f \in L^{6-\frac{8}{p}}(B_r \times [t_1,t_2])$ for any $p>1$, in particular, for $p=8$, we have $\nabla^3 df, \partial_t f \in L^{5}(B_r \times [t_1,t_2])$.
Repeating this process will lead to $\nabla^3 df, \partial_t f \in L^p(B_r \times [t_1,t_2])$ for any $p>1$.

Next, from \eqref{e^pu}, we obtain
\[
\begin{split}
|\nabla u| (t) \leq& C \left( e^{4u}(t_1) |\nabla u|(t_1) + \int_{t_1}^{t} (|\nabla^2 df| |\nabla df| + |\nabla df| |df|^3)(s) ds \right)\\
\int_{B_r} |\nabla u|^p(t) \leq& C \int_{B_r} |\nabla u|^p(t_1) + C (t-t_1)^{p-1} \int_{t_1}^{t} \int_{B_r} (|\nabla^2 df|^p |\nabla df|^p + |\nabla df|^p |df|^{3p})
\end{split}
\]
hence $\nabla u \in L^p(B_r \times [t_1,t_2])$ for any $p>1$, where the bound depends on $\|\nabla u\|_{L^p(B_r)}(t_1)$ and $p$.
Now taking $\nabla$ in the first equation in \eqref{eq1}, we get
\[
|(\partial_t + e^{-4u} \Delta^2) \nabla f| \leq C \left( |\nabla u| |B| + |\nabla B| \right) \in L^p(B_r \times [t_1,t_2])
\]
for any $p>1$, which implies $\nabla^4 df, \partial_t (df) \in L^p(B_r \times [t_1,t_2])$ for any $p>1$.
By repeating this process, we obtain $\nabla^5 df, \nabla^6 df, \partial_t (\nabla df), \partial_t (\nabla^2 df) \in L^p(B_r \times [t_1,t_2])$.

Finally, from Sobolev embedding, we have $f, df, \nabla df, \nabla^2 df \in C^{\alpha}(B_r \times [t_1,t_2])$ for some $\alpha>0$.
This implies $(\partial_t + e^{-4u} \Delta^2) f \in C^{\alpha,\alpha/4}(B_r \times [t_1,t_2])$ where $C^{\alpha,\alpha/4}$ is parabolic H{\"o}lder space of exponent $\alpha$.
By Schauder estimate and standard bootstrapping argument, we complete the proof.
\end{proof}

\section{Short time existence}
\label{sec6}

In this section, we derive short time existence result.
Most of the method and techniques are similar to \cite{P22} or \cite{P25}, and notations are from \cite{MM12}.
Note that, similar to CHF or $n$-CHF, uncoupling of the system \eqref{eq1} using \eqref{e 4u} yields an evolution equation
\[
f_t = \frac{e^{4at}\left( -\Delta^2 f + \Delta (A(df,df)) - \langle \Delta f, \Delta P \rangle + 2 \nabla \langle \Delta f, \nabla P \rangle \right)}{1 + 4b \int_{0}^{t} e^{4as} (|\nabla df|^2 + |df|^4) (s) ds}
\]
which is non-local in time, which is very difficult to analyze.
Instead, we use fixed-point methods in Mantegaza-Martinazzi \cite{MM12} and Park \cite{P22}, \cite{P25}.

%In most of the arguments in this section, we consider $N \hookrightarrow \mathbb{R}^L$ isometrically.
Define spaces $P^m(M,T)$ be the completion of $C^\infty(M \times [0,T],\mathbb{R}^L)$ under the norm
\[
\|f\|_{P^m(M,T)}^2 := \sum_{4j + k \leq 4m} \int_{M \times [0,T]} |\partial_t^j \nabla^k f|^2 dx dt.
\]
Also define the space $Z^m(M,T)$ be the completion of $C^\infty(M \times [0,T])$ under the norm
\[
\|u\|_{Z^m(M,T)}^2 := \sum_{4j + k \leq 4m-2 }\int_{M \times [0,T]} |\partial_t^j \nabla^k u|^2 dx dt.
\]
If the operator $\partial_t + e^{-4u} \Delta^2$ is uniformly parabolic, then the map
\begin{equation}
\Phi : P^m(M,T) \to W^{2(2m-1),2}(M) \times P^{m-1}(M,T), \quad \Phi(f) = (f_0, (\partial_t + e^{-4u} \Delta^2)f)
\end{equation}
is a $C^1$ mapping for $m > \frac{14}{8}$. (Lemma 2.5 in \cite{MM12})
Hence we consider the case $m=2$, that means, we consider $f \in P^2(M,T)$ and $u \in Z^2(M,T)$.
By Sobolev embedding, there exists a universal constant $C$ such that
\[
\begin{split}
 \|\nabla^7 f\|_{L^{\frac{8}{3}}(M \times [0,T])}, \|\nabla^3 f_t\|_{L^{\frac{8}{3}}(M \times [0,T])}, \|\nabla^6 f\|_{L^4(M \times [0,T])}, \|\nabla^2 f_t\|_{L^4(M \times [0,T])}, & \\ 
 \|\nabla^5 f\|_{L^8(M \times [0,T])}, \|\nabla f_t\|_{L^8(M \times [0,T])} , \|\nabla^4 f\|_{L^p(M \times [0,T])}, \| f_t\|_{L^p(M \times [0,T])} \leq& C \|f\|_{P^2}
 \end{split}
 \]
 for any $p>1$, and
 \[
 \begin{split}
 \|f\|_{C^0(M \times [0,T])} , \|\nabla f\|_{C^0(M \times [0,T])}, \|\nabla^2 f\|_{C^0(M \times [0,T])}, \|\nabla^3 f\|_{C^0(M \times [0,T])} \leq& C \|f\|_{P^2}.
 \end{split}
 \]
Similarly, $u$ satisfies
\[
\begin{split}
 \|\nabla^5 u\|_{L^{\frac{8}{3}}(M \times [0,T])}, \|\nabla u_t\|_{L^{\frac{8}{3}}(M \times [0,T])}, \|\nabla^4 u\|_{L^4(M \times [0,T])}, \|u_t\|_{L^4(M \times [0,T])}, & \\ 
 \|\nabla^3 u\|_{L^8(M \times [0,T])} , \|\nabla^2 u\|_{L^p(M \times [0,T])} \leq& C \|u\|_{Z^2}
\end{split}
\]
for any $p>1$, and
\[
 \|u\|_{C^0(M \times [0,T])} , \|\nabla u\|_{C^0(M \times [0,T])}  \leq C \|u\|_{Z^2}.
\]
Also, by standard parabolic theory, we have
\[
\begin{split}
\sup_{0 \leq t \leq T} \|f(t)\|_{W^{6,2}(M)} , \sup_{0 \leq t \leq T} \|f_t(t)\|_{W^{2,2}(M)} \leq& C \|f\|_{P^2},\\
\sup_{0 \leq t \leq T} \|u(t)\|_{W^{4,2}(M)}   \leq& C \|u\|_{Z^2}.
\end{split}
\]
%Note that for $n=3,4$, Sobolev embedding implies that for any $f \in W^{2,2}(M)$ and for any $q>1$,
%\[
%\|f\|_{L^q(M)} \leq C\| f\|_{W^{1,4}(M)} \leq C \| f\|_{W^{2,2}(M)}.
%\]

Fix $f_0 \in W^{6,2}(M,\mathbb{R}^L)$.
If $\|u\|_{C^0} \leq 1$, then the operator $\partial_t + e^{-4u}\Delta^2$ is uniformly parabolic.
Hence there exists $c_1>0$ such that for each $g \in P^1$, there exists a unique solution $h \in P^2$ of the equation
\begin{equation}
(\partial_t + e^{-4u} \Delta^2) h = g, \quad h(0)=f_0
\end{equation}
with
\begin{equation} \label{c_1}
\|h\|_{P^2} \leq c_1 (\|f_0\|_{W^{6,2}} + \|g\|_{P^1}).
\end{equation}

Let $h_0$ be the unique solution of the equation
\[
(\partial_t + \Delta^2) h_0 = 0, \quad h(0) = f_0.
\]
By \eqref{c_1}, there exists a constant $c_2$ only depending on $c_1$ and $\|f_0\|_{W^{6,2}}$ such that
\begin{equation}
\|h_0\|_{P^2} \leq c_2.
\end{equation}

Now for $\delta>0$, consider the closed ball $B_\delta$ in $P^2$
\begin{equation} \label{B_delta}
B_\delta = \{f \in P^2 : f(0) = f_0 \text{ and } \| f - h_0\|_{P^2} \leq \delta\}.
\end{equation}
Also, for $\delta' > 0$, consider the closed ball $\tilde{B}_{\delta'}$ in $Z^2$
\begin{equation} \label{B_delta '}
\tilde{B}_{\delta'} = \{u \in Z^2 : u(0) = 0 \text{ and } \|u \|_{Z^2} \leq \delta'\}.
\end{equation}
Here we assume $\delta'$ small enough so that $\|u\|_{Z^2} \leq \delta'$ implies $\|u\|_{C^0} \leq 1$.
We also assume $\delta ' \leq \frac{\delta}{2 c_1 c_3}$ where $c_3$ is given in \Cref{necessary}.

We are going to construct two operators $S_1 : P^2 \times Z^2 \to P^2$ and $S_2 : P^2 \times Z^2 \to Z^2$ as follows.
Define $S_1(f,u) = h$ where $h \in P^2$ is the unique solution of
\begin{equation} \label{S_1}
(\partial_t + e^{-4u} \Delta^2 ) h = e^{-4u} \left( \Delta (A(df,df)) - \langle \Delta f, \Delta P \rangle + 2 \nabla \langle \Delta f, \nabla P \rangle \right), \quad  h(0) = f_0.
\end{equation}
Similarly, define $S_2(f,u) = v$ where $v \in Z^2$ is the unique solution of
\begin{equation} \label{S_2}
\partial_t v = b e^{-4u} (|\nabla df|^2 + |df|^4) - a, \quad v(0)=0.
\end{equation}

\begin{prop} \label{B to B}
For $f \in P^2$, $u \in Z^2$, $S_1(f,u) \in P^2$ and $S_2(f,u) \in Z^2$.

Moreover, there exists $T_1 = T_1(c_2,\delta,\delta')>0$ such that $S_1$ restricts to $S_1 : B_\delta \times \tilde{B}_{\delta'} \to B_\delta$ and $S_2$ restricts to $S_2 : B_\delta \times \tilde{B}_{\delta'} \to \tilde{B}_{\delta'}$.
\end{prop}

The following lemma is needed in the proof.

\begin{lemma} \label{necessary}
There exists $T_2 = T_2(c_2,\delta,\delta')>0$ such that for all $T \leq T_2$, for any $h \in B_\delta$ and for each $u_1,u_2 \in \tilde{B}_{\delta'}$,
\begin{equation}
\|(e^{4u_1-4u_2}-1) \partial_t h\|_{P^1} \leq c_3 \|u_1-u_2\|_{Z^2}
\end{equation}
for some constant $c_3$ only depending on $c_2$ and $\delta$.
\end{lemma}

\begin{proof}
Observe that $|e^{4u_1-4u_2}-1 | \leq C |u_1-u_2|$ for some constant $C$ if $\|u_1\|_{C^0},\|u_2\|_{C^0} \leq 1$.
Also, for $u \in Z^2$, with $\|u\|_{C^0} \leq 2$, we have
\[
\begin{split}
|\nabla^4 (e^u)| \leq& e^u |\nabla^4 u| + C \left( |\nabla^3 u| |\nabla u| + |\nabla^2 u|^2 + |\nabla^2 u| |\nabla u|^2 + |\nabla u|^4 \right)
\end{split}
\]
hence for all $T$ small enough, $\| \nabla^4 (e^u)\|_{L^2(M \times [0,T])} \leq C(\|u\|_{C^0}) \|u\|_{Z^2}$.
Similar argument gives that for all $T$ small enough, we have $\|\partial_t (e^u)\|_{L^2(M \times [0,T])} \leq C(\|u\|_{C^0}) \|u \|_{Z^2}$ for all $T$ small enough.

Consider
\[
\begin{split}
|\nabla^4 \left( (e^{4u_1-4u_2}-1) \partial_t h \right) | \leq& C \left( |\nabla^4 e^{4(u_1-u_2)}| |\partial_t h| \right.\\
& \left. + |\nabla^3 e^{4(u_1-u_2)}| |\partial_t \nabla h|  + |\nabla^2 e^{4(u_1-u_2)}| |\partial_t \nabla^2 h| \right.\\
& \left. + |\nabla e^{4(u_1-u_2)}| |\partial_t \nabla^3 h| + |u_1-u_2| |\partial_t \nabla^4 h|  \right).
\end{split}
\]
Hence, we get
\[
\begin{split}
\| \nabla^4 \left( (e^{4u_1-4u_2}-1) \partial_t h \right) \|_{L^2(M \times [0,T])}^2 \leq& C(\|h\|_{P^2})  \|u_1-u_2\|_{Z^2}^2 T\\
& +  C \|u_1-u_2\|_{C^0}^2 \|h\|_{P^2}^2\\
\leq&  C c_2 \|u_1-u_2\|_{Z^2}^2 
\end{split}
\]
if $T$ is small enough.
Similar argument will complete the proof.
\end{proof}

\begin{proof}
(Proof of \Cref{B to B})

First, check $S_1(f,u) \in P^2$.
It is enough to show that $e^{-4u} B \in P^1$ where $B = \Delta (A(df,df)) - \langle \Delta f, \Delta P \rangle + 2 \nabla \langle \Delta f, \nabla P \rangle$.
As above, we have
\[
\begin{split}
|B| \leq & C \big( |\Delta df| |df| + |\nabla df|^2 + |\nabla df| |df|^2 + |df|^4 + |\Delta f|^2 + |\Delta f| |df|^2 \\
& \qquad + |\nabla \Delta f| |df| + |\nabla df| |\Delta f| \big).
\end{split}
\]
Hence
\begin{equation} \label{g est}
\begin{split}
\| \nabla^4 (e^{-4u} B)\|_{L^2(M \times [0,T])}^2 \leq& C(\|f\|_{P^2}, \|u\|_{Z^2}) \sum_{k_1+k_2=4} \| |\nabla^{k_1+3} f| |\nabla^{k_2} u| \|_{L^2(M \times [0,T])}^2\\
\leq&C(\|f\|_{P^2}, \|u\|_{Z^2}) T^{\frac{1}{4}}\\
\|\partial_t (e^{-4u} B) \|_{L^2(M \times [0,T])}^2 \leq& C(\|f\|_{P^2}, \|u\|_{Z^2}) \left( \|\partial_t \nabla^3 f\|_{L^2(M \times [0,T])}^2 + \|\partial_t u\|_{L^2(M \times [0,T])}^2 \right)\\
\leq&C(\|f\|_{P^2}, \|u\|_{Z^2}) T^{\frac{1}{4}}.
\end{split}
\end{equation}
This concludes $S_1(f,u) \in P^2$.
And $v = S_2(f,u) \in Z^2$ comes from the fact that
\begin{equation} \label{v est}
\begin{split}
\| \nabla^6 v \|_{L^2(M \times [0,T])}^2 \leq& C(\|f\|_{P^2}, \|u\|_{Z^2}) \sum_{k_1+k_2=6} \| | \nabla^{k_1+2} f| |\nabla^{k_2} u| \|_{L^2(M \times [0,T])}^2 T\\
\leq& C(\|f\|_{P^2}, \|u\|_{Z^2}) T\\
\|\partial_t \nabla^2 v\|_{L^2(M \times [0,T])}^2 \leq& C(\|f\|_{P^2}, \|u\|_{Z^2})  \sum_{k_1+k_2=2} \| \nabla^{k_1+2} f| |\nabla^{k_2} u| \|_{L^2(M \times [0,T])}^2 \\
\leq& C(\|f\|_{P^2}, \|u\|_{Z^2}) T^{1-\frac{2}{p}}
\end{split}
\end{equation}
for any $p>1$.
This observation also suggests that, for $f \in B_\delta$ and $u \in \tilde{B}_{\delta'}$, and for $T$ small enough, we have $v = S_2(f,u) \in \tilde{B}_{\delta'}$.

Next, denote $h = S_1(f,u)$ and $g = e^{-4u} B$.
$h-h_0$ satisfies
\[
(\partial_t + e^{-4u} \Delta^2) (h - h_0) = g + (e^{-4u}-1) \Delta^2 h_0, \quad  (h-h_0)(0) = 0.
\]
Hence by \eqref{c_1}, we get
\[
\|h-h_0\|_{P^2} \leq c_1 \|g + (e^{-4u}-1) \Delta^2 h_0\|_{P^1} \leq c_1 \|g\|_{P^1} + c_1 \|(e^{-4u}-1) \partial_t h_0\|_{P^1}.
\]

Above, we showed that for all $T$ small enough,
\[
\| \nabla^4 g\|_{L^2(M \times [0,T])}, \|\partial_t g\|_{L^2(M \times [0,T])} \leq \frac{\delta}{4c_1}.
\]
For the second term, from \Cref{necessary},
\[
\|(e^{-4u}-1) \partial_t h_0\|_{P^1} \leq c_3 \|u\|_{Z^2} \leq \frac{\delta}{2c_1}.
\]
This completes the proof.
\end{proof}

Next, we show that difference in inputs of $S_1,S_2$ can be controlled by the corresponding norms in $P^2$ or $Z^2$.
More precisely, we show the following proposition.

\begin{prop} \label{S est}
There exists $T_3 = T_3(c_2,\delta,\delta')>0$ such that for all $T \leq T_3$ the followings hold.
\begin{enumerate}
\item
For any $f \in B_\delta$ and $u_1,u_2 \in \tilde{B}_{\delta'}$, we have
\begin{align}
\|S_1(f,u_1)-S_1(f,u_2)\|_{P^2} \leq& c_1 c_3 \|u_1-u_2\|_{Z^2} \label{est 1}\\
\|S_2(f,u_1)-S_2(f,u_2)\|_{Z^2} \leq& \frac{1}{3} \|u_1-u_2\|_{Z^2}. \label{est 2}
\end{align}
\item
For any $f_1,f_2 \in B_\delta$ and $u \in \tilde{B}_{\delta'}$, we have
\begin{align}
\|S_1(f_1,u)-S_1(f_2,u)\|_{P^2} \leq& \frac{1}{3} \|f_1-f_2\|_{P^2} \label{est 3}\\
\|S_2(f_1,u)-S_2(f_2,u)\|_{Z^2} \leq& T^{\frac{1}{2}} \|f_1-f_2\|_{P^2}. \label{est 4}
\end{align}
\end{enumerate}
\end{prop}

\begin{proof}

{\underline {Case (1)}:}

Fix $f \in B_\delta$ and $u_1,u_2 \in \tilde{B}_{\delta'}$.
Denote $h_1 = S_1(f,u_1), h_2 = S_1(f,u_2), v_1 = S_2(f,u_1), v_2 = S_2(f,u_2)$.
Then $h_1-h_2$ satisfies
\[
(\partial_t + e^{-4u_1} \Delta^2) (h_1-h_2) = (e^{4u_1-4u_2}-1) \partial_t h_2, (h_2-h_1)(0) = 0.
\]
Then by \eqref{c_1} and \Cref{necessary},
\[
\|h_1-h_2\|_{P^2} \leq c_1 \|(e^{4u_1-4u_2}-1) \partial_t h_2\|_{P^1} \leq c_1 c_3 \|u_1-u_2\|_{Z^2}
\]
which shows \eqref{est 1}.

$v_1-v_2$ satisfies
\[
\begin{split}
v_1-v_2 =& b \int_{0}^{t} e^{-4(u_1-u_2)} (|\nabla df|^2 + |df|^4)(s) ds.
\end{split}
\]
Hence, similar to \eqref{v est},
\[
\begin{split}
\|\nabla^6 (v_1-v_2)\|_{L^2(M \times [0,T])}^2 \leq& C( \delta, \delta') \sum_{k_1+k_2 = 6} \| |\nabla^{k_1+2} f| |\nabla^{k_2} (u_1-u_2)| \|_{L^2(M \times [0,T])}^2 T\\
\leq& C(\delta, \delta')\|f\|_{P^2}^2 \|u_1-u_2\|_{Z^2}^2 T\\
\|\partial_t \nabla^2 (v_1-v_2)\|_{L^2(M \times [0,T])}^2 \leq& C( \delta, \delta') \sum_{k_1+k_2 = 2} \| |\nabla^{k_1+2} f| |\nabla^{k_2} (u_1-u_2)| \|_{L^2(M \times [0,T])}^2 \\
\leq& C(\delta, \delta')  \|f\|_{P^2}^2 \|u_1-u_2\|_{Z^2}^2 T^{1 - \frac{2}{p}}
\end{split}
\]
for any $p>1$, which implies that for all $T$ small enough, \eqref{est 2} holds.

{\underline {Case (2)}:}

Fix $f_1,f_2 \in B_\delta$ and $u \in \tilde{B}_{\delta'}$.
Denote $h_1 = S_1(f_1,u), h_2 = S_1(f_2,u), v_1 = S_2(f_1,u), v_2 = S_2(f_2,u)$.
Then $h_1-h_2$ satisfies
\[
\begin{split}
(\partial_t + e^{-4u}\Delta^2) (h_1-h_2) =& e^{-4u} \big( \Delta (A(f_1)(df_1,df_1)) - \langle \Delta f_1, \Delta P(f_1) \rangle + 2 \nabla \langle \Delta f_1, \nabla P(f_1) \rangle \big)\\
&- e^{-4u} \big( \Delta (A(f_2)(df_2,df_2)) - \langle \Delta f_2, \Delta P(f_2) \rangle + 2 \nabla \langle \Delta f_2, \nabla P(f_2) \rangle \big)
\end{split}
\]
and $(h_1-h_2)(0) = 0$.
Now the right-hand side is $I + II + III + IV + V + VI$, where
\[
\begin{split}
I =& e^{-4u} \Delta \big( (A(f_1)-A(f_2)) (df_1,df_1) \big)\\
II =& e^{-4u} \Delta \big( A(f_2) (df_1+df_2, df_1-df_2) \big)\\
III =& e^{-4u} \langle \Delta (f_1-f_2),\Delta P(f_1) \rangle \\
IV =& e^{-4u} \langle \Delta f_2, \Delta (P(f_1) - P(f_2)) \rangle\\
V =& e^{-4u} 2 \nabla \langle \Delta (f_1-f_2), \nabla P(f_1) \rangle\\
VI =& e^{-4u} 2 \nabla \langle \Delta f_2, \nabla (P(f_1)-P(f_2)) \rangle.
\end{split}
\]
We claim that $\|I\|_{P^1}, \ldots, \|VI\|_{P^1} \leq \frac{1}{18 c_1} \|f_1-f_2\|_{P^2}$ if $T$ is small enough.
To see this, note that $|A(f_1)-A(f_2)| \leq C(\delta) |f_1-f_2|$ and so
\[
\begin{split}
\|\nabla^4 I \|_{L^2(M \times [t_1,t_2])}^2  \leq& C(\delta, \delta') \Bigg( \sum_{k_1+k_2+k_3 = 4} \| |\nabla^{k_1+2} (f_1-f_2)| |\nabla^{k_2+1} f_1 | |\nabla^{k_3} u| \|_{L^2(M \times [t_1,t_2])}^2 \\
&\qquad \qquad +   \| |\nabla^{k_1+1} (f_1-f_2)| |\nabla^{k_2+2} f_1 | |\nabla^{k_3} u| \|_{L^2(M \times [t_1,t_2])}^2 \\
&\qquad \qquad +  \| |\nabla^{k_1} (f_1-f_2)| |\nabla^{k_2+3} f_1 | |\nabla^{k_3} u| \|_{L^2(M \times [t_1,t_2])}^2 \Bigg)\\
\leq& C(\delta, \delta') \|f_1-f_2\|_{P^2}^2 T^{\frac{1}{4}}
\end{split}
\]
\[
\begin{split}
\|\partial_t I\|_{L^2(M \times [t_1,t_2])}^2 \leq& C(\delta, \delta') \Bigg( \| |\partial_t (f_1-f_2)| |\nabla f_1 | | u| \|_{L^2(M \times [t_1,t_2])}^2 \\
& \qquad \qquad + \| |(f_1-f_2)| |\partial_t \nabla f_1 | | u| \|_{L^2(M \times [t_1,t_2])}^2 \\
& \qquad \qquad +  \| |(f_1-f_2)| |\nabla f_1 | | \partial_t u| \|_{L^2(M \times [t_1,t_2])}^2 \Bigg)\\
\leq& C(\delta,\delta') \|f_1-f_2\|_{P^2}^2 T^{\frac{3}{4}}.
\end{split}
\]
Similar arguments can conclude the claim, and hence complete \eqref{est 3}.

$v_1-v_2$ satisfies
\[
\begin{split}
v_1 - v_2 =& b \int_{0}^{t} e^{-4u} (|\nabla df_1|^2 + |df_1|^4 - |\nabla df_2|^2 - |df_2|^4)(s) ds\\
=& b \int_{0}^{t} e^{-4u} ( \langle \nabla^2 (f_1-f_2), \nabla^2 (f_1+f_2) \rangle + (|df_1|^2 + |df_2|^2) \langle \nabla (f_1-f_2), \nabla (f_1+f_2) \rangle) (s)ds.
\end{split}
\]
Hence, similar to \eqref{v est},
\[
\begin{split}
\|\nabla^6 (v_1-v_2)\| &_{L^2(M \times [0,T])}^2\\
 \leq& C( \delta, \delta') \sum_{k_1+k_2+k_3 = 6} \| |\nabla^{k_1+2} (f_1-f_2)| |\nabla^{k_2+2} (f_1+f_2)| |\nabla^{k_3} u| \|_{L^2(M \times [0,T])}^2 T\\
\leq& C(\delta, \delta')\|f_1-f_2\|_{P^2}^2 \|u\|_{Z^2}^2 T,\\
\|\partial_t \nabla^2 (v_1-v_2)\|&_{L^2(M \times [0,T])}^2\\
 \leq& C( \delta, \delta') \sum_{k_1+k_2+k_3 = 2} \| |\nabla^{k_1+2} (f_1-f_2)| |\nabla^{k_2+2} (f_1+f_2)| |\nabla^{k_3} u| \|_{L^2(M \times [0,T])}^2\\
\leq& C(\delta, \delta') \|f_1-f_2\|_{P^2}^2 \|u\|_{Z^2}^2 T^{1-\frac{2}{p}}
\end{split}
\]
for any $p>1$, which implies that for all $T$ small enough, \eqref{est 4} holds.
\end{proof}

Now we are ready to show short-time existence.
First, we define a Banach space $X = P^2 \times Z^2$ equipped with the norm
\begin{equation}
\|(f,u)\|_{X} = (2c_1c_3)^{-1} \|f\|_{P^2} + \|u\|_{Z^2}.
\end{equation}

\begin{theorem} \label{short time existence}
Fix $f_0 \in W^{6,2}(M,N)$.
Then there exists $T_0 = T_0(c_2,\delta,\delta')>0$ such that a solution $(f,u) \in B_\delta \times \tilde{B}_{\delta'} \subset P^2 \times Z^2$ of \eqref{eq1} exists on $M \times [0,T_0]$ and $f(M \times [0,T_0]) \subset N$.
\end{theorem}

\begin{proof}
We let $T_0 = T_0(c_1,c_2,c_3, \delta,\delta')$ by
\begin{equation}
T_0 = \min \left\{ T_1,T_2,T_3, \left( \frac{1}{2} (2 c_1 c_3)^{-1} \right)^2 \right\} > 0.
\end{equation}
Now consider the operator $\mathcal{S} : X \to X$ given by $\mathcal{S} (f,u) = (S_1(f,u),S_2(f,u))$.
From \Cref{B to B}, since $T_0 \leq T_1$, $\mathcal{S}$ restricts to $\mathcal{S} : B_\delta \times \tilde{B}_{\delta'} \to B_\delta \times \tilde{B}_{\delta'}$.

Let $f_1,f_2 \in B_\delta$ and $u_1,u_2 \in \tilde{B}_{\delta'}$.
From \Cref{S est}, for any $T \leq T_0$, we have
\[
\begin{split}
\|\mathcal{S}(f_1,u_1)& - \mathcal{S}(f_2,u_2)\|_{X}\\
 =& (2c_1 c_3)^{-1} \|S_1(f_1,u_1) - S_1(f_2,u_2)\|_{P^2} + \|S_2(f_1,u_1) - S_2(f_2,u_2)\|_{Z^2}\\
\leq&(2c_1 c_3)^{-1} \|S_1(f_1,u_1) - S_1(f_1,u_2)\|_{P^2} + (2c_1 c_3)^{-1} \|S_1(f_1,u_2) - S_1(f_2,u_2)\|_{P^2}\\
& + \|S_2(f_1,u_1) - S_2(f_1,u_2)\|_{Z^2} + \|S_2(f_1,u_2) - S_2(f_2,u_2)\|_{Z^2}\\
\leq& \frac{1}{2} \|u_1 -u_2\|_{Z^2} + (2c_1 c_3)^{-1}  \frac{1}{3} \|f_1-f_2\|_{P^2} + \frac{1}{3} \|u_1-u_2\|_{Z^2} + T^{\frac{1}{2}} \|f_1-f_2\|_{P^2}\\
\leq& \frac{5}{6} \Big( (2c_1 c_3)^{-1} \|f_1-f_2\|_{P^2} + \|u_1-u_2\|_{Z^2} \Big)\\
 =& \frac{5}{6} \|(f_1,u_2)-(f_2,u_2)\|_{X}.
\end{split}
\]
By Banach fixed point theorem, there exists $(f,u) \in P^2 \times Z^2$ such that $S_1(f,u) = f$ and $S_2(f,u)=u$.
This shows that $(f,u)$ solves \eqref{eq1}.

Finally, we need to show that $f(M \times [0,T_0]) \subset N$.
As $N$ be a smooth Riemannian manifold isometrically embedded in $\mathbb{R}^L$, we can find its tubular neighborhood $N_\delta$ and consider the nearest point projection $\Pi_N : N_\delta \to N$.
Note that $P = \nabla \Pi(y) : \mathbb{R}^L \to T_y N$ is the orthogonal projection and $A(y) = -\nabla^2 \Pi(y) : T_y N \otimes T_y N \to (T_y N)^{\perp}$ for $y \in N$ is the second fundamental form of $N \hookrightarrow \mathbb{R}^L$.
Define
\begin{equation}
\rho(x,t) = \left| \Pi(f(x,t)) - f(x,t) \right|^2.
\end{equation}
Then $\rho(\cdot,0) = 0$ and by direct computation, we have
\[
\begin{split}
\partial_t \rho  = & 2 \langle \Pi(f)-f, \nabla \Pi (f_t) - f_t \rangle\\
=& - 2 e^{-4u} \langle \Pi(f)-f, \Delta^2 f - \Delta (A(df,df)) + \langle \Delta f, \Delta P \rangle - 2\nabla \langle \Delta f, \nabla P \rangle \rangle\\
=& -2 e^{-4u} \langle \Pi(f)-f, P(\Delta^2 f) \rangle = 0
\end{split}
\]
because $\Pi(f)-f \perp T_f N$ and $P(\Delta^2 f) \in T_f N$.
Hence, $\rho(\cdot, t) = 0$ for all $t$ and $f(M \times [0,T_0]) \subset N$.
\end{proof}

%%%%%%%%%%%%%%%%%%%%%%%%%%%%%%%%%%%%%%%%%%%%%%%%%%%%%%%%%%%%%%%%%%%%%%%%%%%%%
\section{Global regularity}
\label{sec7}

From \Cref{sec6}, there exists $T_0>0$ such that the smooth solution $(f,u)$ of \eqref{eq1} exists on $M \times [0,T_0)$.
Assume $T_0$ be the maximal time such that the solution $(f,u)$ is smooth on $M \times [0,T_0)$.
Also assume that $T_0 < \infty$, that is, finite time singularity exists.
We first show that the criterion for finite time singularity is energy concentration.
Then we show that such energy concentration cannot be happened, concluding that there is no finite time singularity.

First we show that the solution $(f,u)$ of \eqref{eq1} obtained in \Cref{short time existence} is smooth.
\begin{prop} \label{smooth}
Let $(f,u) \in P^2 \times Z^2$ be a solution of \eqref{eq1} on $M \times [0,T]$.
Then $(f,u)$ is smooth on $M \times [0,T]$.
\end{prop}

\begin{proof}
Since $f \in P^2$, we have $\|df\|_{C^0}, \|\nabla df\|_{C^0} \leq C \|f\|_{P^2} =: M < \infty$.
Then
\[
e^{-4u} = \frac{e^{4at}}{1 + 4b \int_{0}^{t} e^{4as} (|\nabla df|^2 + |df|^4) (s) ds} \geq \frac{1}{1 + 2\frac{b}{a} {M}^{4}}.
\]
So, the operator $\partial_t + e^{-4u} \Delta^2 $ is uniformly parabolic, hence by bootstrapping argument with $\nabla^2 df \in C^{\alpha,\alpha/2}$, $f$ is smooth, hence $u$ is also smooth on $M \times [0,T)$.
\end{proof}

Next, we develope the global version of \Cref{Der p=0}.
\begin{prop} \label{Der p=0 glob}
(Derivative estimate, global version)
Let $(f,u)$ be a smooth solution of \eqref{eq1} on $M \times [t_1,t_2]$.
Then
\begin{equation}
\begin{split}
\frac{d}{dt}  \int_{M} e^{4u}|f_t|^2  \leq& 4a \int_{M} e^{4u}|f_t|^2 -  \int_{M} |\Delta f_t|^2 -C_b \int_{M} |f_t|^2 (|\nabla df|^2 + |df|^4) .
\end{split}
\end{equation}
\end{prop}

Its proof is almost the same as \Cref{Der p=0}.
As a result of \Cref{Der p=0 glob} and \Cref{E dec}, we have
\[
\int_{M} e^{4u}|f_t|^2  (t_2) - \int_{M} e^{4u}|f_t|^2 (t_1) \leq 4a \int_{t_1}^{t_2} \int_{M} e^{4u}|f_t|^2 = 4a (E(t_1) - E(t_2))
\]
hence for any $t \geq 0$,
\begin{equation} \label{int f_t^2 bound}
\int_{M} e^{4u}|f_t|^2 (t) \leq 4a E(0).
\end{equation}

\begin{theorem} \label{finitely many}
Let $(f,u)$ be a smooth solution of \eqref{eq1} on $M \times [0,T_0)$.
Assume $T_0<\infty$ is the maximal existence time.
Then there exists at most finitely many points $x_1, \cdots, x_k$ such that for all $i=1, \cdots, k$,
\begin{equation} \label{finite time sing criterion}
\lim_{r \to 0} \limsup_{t \nearrow T_0} \int_{B_r(x_i)} |\nabla df|^2 + |df|^4 > \ep_1.
\end{equation}
\end{theorem}

\begin{proof}
We first show that if for $x \in M$,
\begin{equation} \label{small E}
\lim_{r \to 0} \limsup_{t \nearrow T_0} \int_{B_r(x)} |\nabla df|^2 + |df|^4 \leq \ep_1,
\end{equation}
then $f$ and $u$ is smooth at $(x,T_0)$.

From above assumption, we may assume that for some fixed $t_1$ with $0 < t_1 < T_0$, for all $r>0$ small enough, $\sup_{[t_1,T_0]} \int_{B_r(x)}  |\nabla df|^2 + |df|^4 (t) \leq \ep_1$.
%Since $f$ is smooth at $t_1$, we choose $0 < r_1 $ such that $\int_{B_{r_1}(x)}  |\nabla df|^2 + |df|^4(t_1) \leq \ep_1$.
Also, since $u$ is smooth at $t_1$, $V_u(t_1) = \sup_{B_r} u(x,t_1) < \infty$.
Then by \Cref{loc smooth}, we obtain that $f$ is smooth and $u$ is also smooth at $(x,T_0)$.

Next, assume that for all $x \in M$, \eqref{small E} holds.
Then $f,u$ are smooth on $M \times \{T_0\}$, hence by \Cref{short time existence}, the solution $(t,u)$ exists on $M \times [0,T_0+\delta]$ for some $\delta>0$.
This conflicts with the assumption that $T_0$ is the maximal existence time.
Therefore, there should be $x \in M$ such that for any $r>0$,
\begin{equation} \label{large E}
\limsup_{t \nearrow T_0} \int_{B_r(x)} |\nabla df|^2 + |df|^4 > \ep_1.
\end{equation}

We show that there are at most finitely many such points.
Let $x_1, \cdots, x_k$ be any finite collection of such points.
Fix $R>0$ so that $B_{2R}(x_i)$ are disjoint.
%Let $T_0' = T_0 - \frac{\bar{\ep}}{4 C_3 \left( 1 + \frac{1}{R^n} \right)^2}$.
Let $T_0'<T_0$ and let $0 < r \leq R$.
%Here without loss of generality, we may assume $\frac{r}{R} < \ep$, where $\alpha = \alpha(T_2,R)$ is from \Cref{df^4 lem}.
%Choose $t_i$ such that $T_0' \leq t_i \leq T_0$ and
%\[
% \int_{B_r(x_i)} |\nabla df|^2 + |df|^4 (t_i) > \frac{\ep_1}{2}.
%\]

We first claim that
\begin{equation}
\sup_{T_0' \leq t \leq T_0} \int_{B_R(x_i)} |\Delta f|^2 (t)  >  s \ep_1
\end{equation}
where $s>0$ is chosen such that
\[
s \leq \min \{ \ep_1^{-1}, (8C_8)^{-1} \}.
\]
Suppose $\sup_{T_0' \leq t \leq T_0} \int_{B_R(x_i)} |\Delta f|^2 (t) \leq  s \ep_1$.
By \eqref{df^4 nabla df^2 ep}, for $t$ close enough to $T_0$,
\[
\begin{split}
\frac{\ep_1}{2} < \int_{B_r(x_i)} |\nabla df|^2 + |df|^4 (t) \leq&  C_8 (s\ep_1)^2 + C_8 s\ep_1 \leq 2C_8 s \ep_1 \leq \frac{\ep_1}{4}
\end{split}
\]
which is impossible, hence verify the claim.
Note that because of \eqref{large E}, we can also say that
\begin{equation} \label{finite time sing criterion 2}
\limsup_{t \nearrow T_0} \int_{B_R(x_i)} |\Delta f|^2   >  s \ep_1.
\end{equation}
Note that above condition is independent on $R$.
%Let $T_0' = T_0 - \frac{\bar{\ep}}{4 C_3 \left( 1 + \frac{1}{R^n} \right)^2}$.
Let $T_0'<T_0$ to be determined later. 
Choose $t_i \in [T_0',T_0]$ such that
\[
 \int_{B_R(x_i)} |\Delta f|^2 (t_i) > \frac{s \ep_1}{2}.
\]

From \Cref{loc E finer}, and using  on $B_{r}(x_i)$, we have
\[
\begin{split}
4a \frac{s \ep_1}{2} <& 4a \int_{B_R(x_i)} |\Delta f|^2 (t_i)\\
 \leq& 4a K^i(T_0') +C (T_0-T_0') e^{4aT_0} (1 + \frac{1}{R^4})  \left( K^i(T_0') + C e^{4aT_0} (1 + \frac{1}{R^4}) (1+T_0) E(0) \right) \\
 \leq& \left(4a + C(T_0-T_0') e^{4aT_0} (1 + \frac{1}{R^4}) \right) K^i(T_0') + C(T_0-T_0') e^{8aT_0} (1 + \frac{1}{R^4})^2 (1+T_0)E(0)
 \end{split}
 \]
 where
 \[
 K^i(T_0') = \int_{B_{2R}(x_i)} e^{4u}|f_t|^2 (T_0') + 4a \int_{B_{2R}(x_i)} |\Delta f|^2 (T_0').
 \]
 Now we choose $T_0'$ such that
 \[
 C(T_0-T_0') e^{8aT_0} (1 + \frac{1}{R^4})^2 (1+T_0)E(0) \leq 2a \frac{s \ep_1}{2}.
 \]
 This implies
 \[
 as \ep_1  < \left(4a + C(T_0-T_0') e^{4aT_0} (1 + \frac{1}{R^4}) \right) \left(  \int_{B_{2R}(x_i)} e^{4u}|f_t|^2 (T_0') + 4a \int_{B_{2R}(x_i)} |\Delta f|^2 (T_0') \right).
\]
Denote $C_{16} = \left(4a + C(T_0-T_0') e^{4aT_0} (1 + \frac{1}{R^4}) \right)$.
Then we finally have
\[
\begin{split}
 as \ep_1 <& C_{16} \int_{B_{2R}(x_i)} e^{4u}|f_t|^2 (T_0')  + 4a C_{16} \int_{B_{2R}(x_i)} |\Delta f|^2 (T_0') \\
k as \ep_1 <& C_{16} \int_{M} e^{4u}|f_t|^2 (T_0') + 4a C_{16}  \int_{M}  |\Delta f|^2 (T_0'). 
\end{split}
\]
From \Cref{E dec}, $\int_{M} |\Delta f|^2 (T_0') \leq E(0)$.
Also, from \eqref{int f_t^2 bound}, $\int_{M} e^{4u}|f_t|^2 (T_0') \leq 4a E(0)$.
Hence RHS is finite, which implies $k$ is bounded from above.
This completes the proof.
\end{proof}

Next, we show that actually there is no such finite time singularity.

\begin{proof}
(Proof of \Cref{main 1})
From \Cref{sec6}, there exists $T_0>0$ such that the smooth solution $(f,u)$ of \eqref{eq1} exists on $M \times [0,T_0)$.
Assume $(x_0,T)$ be a finite time singularity and fix $B_R(x_0)$ be a ball centered at $x_0$.
The condition \eqref{finite time sing criterion 2} implies
\[
\limsup_{t \nearrow T} \int_{B_R(x_0)} |\Delta f|^2 (t) = K + \int_{B_R(x_0)} |\Delta f|^2 (T)
\]
for some positive constant $K > s \ep_1 > 0$, measuring the energy loss at $(x_0,T)$.
Equivalently, since \eqref{finite time sing criterion 2} is independent on $R$, $K$ can be described by
\[
K = \lim_{r \to 0} \limsup_{t \nearrow T} \int_{B_r(x_0)} |\Delta f|^2 (t).
\]

Let $0 < r \leq R$ and $\varphi$ be a cut-off function on $B_r(x_0)$.
Define the local energy
\begin{equation}
\Theta_r(t) = \int_{B_r(x_0)} |\Delta f|^2 \varphi^4.
\end{equation}
Then
\[
\begin{split}
\frac{\Theta_r(t)}{dt} =  \int_{B_r(x_0)} \varphi^4 \langle \Delta f_t, \Delta f \rangle \leq& \left( \int_{B_r(x_0)} |\Delta f_t|^2 \varphi^4 \right)^{\frac{1}{2}} \left( \int_{B_r(x_0)} |\Delta f|^2 \varphi^4 \right)^{\frac{1}{2}}\\
\leq& E(0)^{\frac{1}{2}} \left( \int_{B_r(x_0)}  |\Delta f_t|^2 \right)^{\frac{1}{2}}.
\end{split}
\]
So, for $s \leq t < T$, by \Cref{delta f_t lem} with $\ep=1$,
\begin{equation}
\begin{split}
\Theta_r(t) - \Theta_r(s) =& E(0)^{\frac{1}{2}} \int_{s}^{t} \left( \int_{B_r(x_0)}  |\Delta f_t|^2 \right)^{\frac{1}{2}}\\
\leq& E(0)^{\frac{1}{2}} (t-s)^{\frac{1}{2}} \left( \int_{s}^{t} \int_{B_R(x_0)} |\Delta f_t|^2 \right)^{\frac{1}{2}}\\
\leq& C_{17} (t-s)^{\frac{1}{2}}
\end{split}
\end{equation}
where the constant $C_{17}$ only depends on $E(0), T, R$.

Hence we can take the limit $\lim_{t \nearrow T} \Theta_r(t)$ and have
\[
K = \lim_{t \nearrow T} \int_{B_r(x_0)} |\Delta f|^2 (t) - \int_{B_r(x_0)} |\Delta f|^2 (T) = \lim_{t \nearrow T} \Theta_r (t) - \Theta_r(T).
\]
Combining above two inequality gives
\begin{equation} \label{Theta est}
|K + \Theta_r(T) - \Theta_r(s)| \leq C_{17} (T-s)^{\frac{1}{2}}.
\end{equation}
Now, fix $\delta >0$ such that
\begin{equation}
\delta = \min\{ \frac{1}{C_{17}}, \frac{K}{4}\}.
\end{equation}
Also fix $s = T - \delta^4$ and choose $r$ small enough so that $\Theta_r(T) < \frac{K}{4}$ and $\Theta_r(s) < \frac{K}{4}$.
Then from \eqref{Theta est} we have
\[
\frac{K}{2} \leq C_{17} (T-s)^{\frac{1}{2}} \leq C_{17} \delta^2 < \delta \leq \frac{K}{4}
\]
which is a contradiction.
Hence there is no finite time singularity.
Then by bootstrapping argument for uniformly parabolic equation, we obtain smooth solution $(f,u)$ of \eqref{eq1} and we complete the proof of \Cref{main 1}.
\end{proof}

%\section*{Acknowledgments}

%The author would like to thank to Thomas Parker for valuable comments and simplifications of the argument.
%The author thanks to the anonymous referee for his careful reading and crucial suggestions to the previous version of this paper.

%This is for bibliography example.
%\cite{HSM29}.
%%The order to proceed is
% 1. pdfLaTex
% 2. BibTex
% 3. pdfLaTex
% 4. pdfLaTex

\bibliographystyle{abbrv}
\bibliography{bib}

\end{document}